\documentclass[reqno]{amsart}

\makeatletter
\@namedef{subjclassname@2020}{\textup{2020} Mathematics Subject Classification}
\makeatother

\usepackage{amssymb}
\usepackage{graphicx}
\usepackage{mathrsfs}
\usepackage{hyperref}
\usepackage{color}
\usepackage{url}
\usepackage{cite}

\usepackage{times}
\usepackage{microtype}
\usepackage{amssymb}
\usepackage{mathtools}
\usepackage{tikz}
\usepackage{wasysym}
\usepackage{centernot}
\usepackage{color}
\usepackage{setspace}
\usepackage{appendix}
\usepackage{booktabs}
\usepackage{multirow}

\def\bP{\mathbf{P}}
\def\bF{\mathbf{F}}

\def\bR{{\mathbb R}}

\def\sE{{\mathscr E}}
\def\sF{{\mathscr F}}
\def\sA{{\mathscr A}}
\def\sG{{\mathscr G}}

\def\cS{\mathcal{S}}
\def\cF{\sF}

\def\cN{\mathcal{N}}

\def\sL{\mathscr{L}}
\def\cD{\mathcal{D}}

\def\bN{\mathbb{N}}

\def\limsup{\mathop{\overline{\mathrm{lim}}}}

\def\${|\!|\!|}
\def\l|{\left|\!\left|\!\left|}
\def\r|{\right|\!\right|\!\right|}

\newtheorem{theorem}{Theorem}[section]
\newtheorem{lemma}[theorem]{Lemma}
\newtheorem{proposition}[theorem]{Proposition}
\newtheorem{corollary}[theorem]{Corollary}
\theoremstyle{definition}
\newtheorem{definition}[theorem]{Definition}

\theoremstyle{remark}
\newtheorem{remark}[theorem]{Remark}
\numberwithin{equation}{section}

\begin{document}

\title[Ray-Knight compactification]{Ray-Knight compactification of birth and death processes}

\author{Liping Li}
\address{Fudan University, Shanghai, China.  }
\address{Bielefeld University,  Bielefeld, Germany.}
\email{liliping@fudan.edu.cn}
\thanks{The author is a member of LMNS, Fudan University.  He is also partially supported by NSFC (No. 11931004 and 12371144) and Alexander von Humboldt Foundation in Germany. }



\date{\today}



\begin{abstract}
A birth and death process is a continuous-time Markov chain with minimal state space $\bN$,  whose transition matrix is standard and whose density matrix is {a} birth-death matrix.  {Birth and death process is} unique if and only if $\infty$ is {an} entrance or natural.  When $\infty$ is neither {an} entrance nor natural,  {there are two ways in the literature} to obtain all birth and death processes.  The first one is an analytic treatment {proposed } by Feller in 1959,  and the second {one} is a probabilistic construction completed by Wang in 1958.  

In this paper we will {give another way to study birth and death processes using the Ray-Knight compactification}.  This way has {the advantage of both the analytic and probabilistic treatments above}.  {By} applying {the} Ray-Knight compactification,  every birth and death process {can be modified into a c\`adl\`ag Ray process} on {$\bN\cup \{\infty\}\cup\{\partial\}$,  which is either a Doob processes or a Feller $Q$-process}.  Every birth and death process in the second class has a modification that is a Feller process on {$\bN\cup\{\infty\}\cup \{\partial\}$}.  We will derive the expression of its infinitesimal generator,  which {explains its boundary behaviours at $\infty$}.   Furthermore,  {by} {using the killing transform and the Ikeda-Nagasawa-Watanabe piecing out procedure}, we will also {provide} a probabilistic construction for birth and death processes.  {This construction relies on a triple determining the resolvent matrix introduced by Wang and Yang in their work \cite{WY92}.}
\end{abstract}

\maketitle

\tableofcontents

\section{Introduction}\label{SEC1}

Throughout this paper, we consider a birth-death density matrix as follows:
\begin{equation}\label{eq:11}
Q=(q_{ij})_{i,j\in \bN}:=\left(\begin{array}{ccccc}
-q_0 &  b_0 & 0 & 0 &\cdots \\
a_1 & -q_1 & b_1 & 0 & \cdots \\
0 & a_2 & -q_2 & b_2 & \cdots \\
\cdots & \cdots &\cdots &\cdots &\cdots 
\end{array}  \right),
\end{equation}
{where $a_k>0$ for $k \geq 1$ and $b_k>0,  q_k=a_k+b_k$ for $k\geq 0$.} (Set $a_0=0$ for convenience.)  As is well known,  a transition matrix $(p_{ij}(t))_{i,j\in \bN}$ is called honest if $\sum_{j\in \bN}p_{ij}(t)=1$ for all $i\in \bN$ and $t\geq 0$.  A non-honest transition matrix can always be treated as an honest one on $\bN\cup \{\partial\}$ by a standard argument,  where $\partial$ is an attached trap point; see,  e.g.,  \cite[II \S3, Theorem~3]{C67}.  The transition matrix is called \emph{standard} if $\lim_{t\rightarrow 0}p_{ij}(t)=\delta_{ij}$ for any $i,j\in\bN$.  A continuous-time Markov chain is called a \emph{birth-death $Q$-process} ($Q$-process for short) if its transition matrix $(p_{ij}(t))_{i,j\in \bN}$ is standard and its \emph{density matrix} is $Q$,  i.e.  $p'_{ij}(0)=q_{ij}$ for $i,j\in \bN$.  This amounts to saying that $\mathbf{P}(t):=(p_{ij}(t))_{i,j\in \bN}$ is standard and satisfies the Kolmogorov backward equation:
\begin{equation}\label{eq:02}
	\bP'(t)=Q\bP(t).
\end{equation}
Two $Q$-processes with the same transition matrix will not be distinguished in our consideration.  For convenience we would also call such $(p_{ij}(t))_{i,j\in \bN}$ a $Q$-process {when there is no danger of confusion}.  

There always exists at least one $Q$-process, called the \emph{minimal $Q$-process}; see,  e.g.,  \cite{F40},  \cite[\S10]{F59}, and also Section \ref{SEC21}.  The minimal $Q$-process is killed at the first \emph{flying time},  which is the first time  it almost reaches $\infty$ (see Corollary~\ref{COR35}).  A well-known uniqueness criterion states that the {$Q$-process is unique  if} and only if the {first flying time} of  the minimal $Q$-process is equal to infinity (namely,  the minimal transition matrix is honest on $\bN$).  In Feller's {terminology on the classification of} the right boundary point $\infty$ in terms of the minimal $Q$-process (see Definition~\ref{DEF23}),  this uniqueness is also equivalent to $\infty$ being \emph{{an entrance}} or \emph{natural} (see Theorem~\ref{THM25}). When $\infty$ is neither {an entrance} nor natural (i.e.,  either \emph{regular} or \emph{an exit}),  there exist other $Q$-processes besides the minimal one.  The most famous examples are {the so-called} \emph{Doob processes} (see,  e.g.,  \cite[\S15]{F59} and \cite[\S6.1]{WY92}).  Roughly speaking,  a Doob process returns to $\bN\cup \{\partial\}$ immediately at each flying time and is uniquely determined by its returning distribution,  which is a probability measure on $\bN\cup \{\partial\}$.  In particular,  it reduces to the minimal $Q$-process if the entering distribution is the point measure at $\partial$.  

It is certainly interesting to {find} all $Q$-processes for the non-trivial case that $\infty$ is regular or an exit.  To our knowledge, {there are two different ways in the literature.} The first way is an analytic treatment that involves constructing the corresponding resolvents.  {This approach} was introduced by Feller \cite{F59},  {who {found} $Q$-processes that satisfy not only \eqref{eq:02} but also the Kolmogorov forward equation}.  Yang later expanded on this work in 1965 to encompass all $Q$-processes (see \cite{WY92}).  {The main result in this approach states} that the resolvent of every $Q$-process is determined by a triple $(\nu, \gamma,\beta)$,  where $\nu$ is a positive Radon measure on $\bN$ and $\gamma$ and $\beta$ are two non-negative constants.  The concrete expression of this resolvent will be reviewed in Appendix~\ref{APPB} for {the readers}' convenience.  Although this argument is simple to follow,  {it does not provide a clear explanation of} the probabilistic meaning of these parameters.  The second way, given by Wang in 1958 (see \cite{WY92}), provides a probabilistic construction for all $Q$-processes. It turns out that every $Q$-process can be approximated by a sequence of Doob processes,  which allows for a characterization of the $Q$-process by certain parameters with specific probabilistic interpretations. This construction is more friendly to probabilists,  {but} the argument is long and somewhat complicated.

The Ray-Knight compactification,  {introduced} by Ray \cite{R59} and firmly established after a revision by Knight \cite{K65},  provides a way to modify an (almost arbitrary) Markov process  in a slight manner {to get} a nice Markov process on a {larger topological} space,  called {a} \emph{Ray process}.  The original process is assumed to enjoy neither {the strong Markov} property {nor sample path regularity beyond Borel measurability. } However,  the resulting Ray process is c\`adl\`ag and strongly Markovian.  Additionally,  the quasi-left-continuity of sample paths also holds in a certain {sense}.  Loosely speaking, a Ray process is analogous to a Feller process, differing only in the potential existence of \emph{branching points} in the state space: if $x$ is a branching point, then $\bar{P}_0(x,\cdot)$ is a probability measure on the set of non-branching points, rather than the Dirac measure $\delta_x$. Here,  $(\bar{P}_t)_{t\geq 0}$ denotes the transition semigroup of {the} Ray process.  
More heuristically,  when the Ray process almost reaches a branching point $x$, it will never arrive at $x$ but will instead “branch” to a non-branching position according to the distribution $\bar{P}_0(x,\cdot)$.  For further details, we refer {the readers} to,  e.g.,  \cite{CW05},  \cite{GR75} and \cite{RW00}.

{In this paper, we will introduce another method to study $Q$-processes using the Ray-Knight compactification. } This approach yields a Ray process on {$\bN\cup \{\infty\}\cup \{\partial\}$}, which is a c\`adl\`ag modification of {the} given $Q$-process.  Particularly,  only the additional point $\infty$ is possibly branching (see Lemma~\ref{LM33}~(2)).  As a result, there are only two essentially different cases: $\infty$ is either branching or non-branching.  In the former case where $\infty$ is branching,  this Ray process (which is not distinguished from the $Q$-process in our context because they have the same transition matrix) is precisely a Doob process (see Theorem~\ref{THM41}).  The latter case is more interesting,  and in this case,  the Ray process becomes a Feller process; see,  e.g.,  \cite[III (37.1)]{RW00}.  We call it a \emph{Feller $Q$-process}.  

With the help of {the} Feller property,  we can {further} characterize Feller $Q$-processes {both analytically and probabilistically}.  {Analytically},  the infinitestimal operator of a Feller $Q$-process is derived in Theorem~\ref{THM62}. The boundary condition \eqref{eq:74} provides insight into the meaning of the triple $(\nu, \gamma, \beta)$ that determines the resolvent: the measure $\nu$ represents the \emph{non-local} jumping intensity from $\infty$ to $\bN$,  $\gamma$ represents the killing intensity at $\infty$,  and $\beta$ represents the reflecting intensity at $\infty$.  Here and throughout, non-local (resp. local) jumps refer to jumps between non-adjacent (resp. adjacent) positions.  Further explanation will be {given} in Section \ref{SEC23-2}. It is also worth noting that $\nu$ plays a similar role to the L\'evy measure for a L\'evy process, although the non-local jumps of a $Q$-process only occur at flying times.  {Probabilistically},  we present a construction of Feller $Q$-processes using {the killing transform and the Ikeda-Nagasawa-Watanabe piecing out procedure} (see \cite{INW66} and also Appendix~\ref{APPA}).  The {starting point} of this construction is a special $Q$-process, called the $(Q,1)$-process in the terminology of \cite{WY92}.  It is the unique honest,  symmetrizable $Q$-process that satisfies both the backward and forward Kolmogorov equations; see,  e.g.,  \cite[Theorem~11.1]{F59}.  (Strictly speaking,  it exists only {when} $\infty$ is regular.)  Then,  a Feller $Q$-process can be constructed pathwisely as follows: break a path of the $(Q,1)$-process at a specific flying time and then piece together a renewal sample path that starts from a randomly chosen position in $\mathbb{N}$ according to $\nu$; {repeat these two procedures indefinitely until the path enters the trap $\partial$}. This construction can be made rigorous when $\nu$ is a finite measure.  However,  if $\nu$ is not a finite measure,  selecting a random position in $\mathbb{N}$ according to $\nu$ may be problematic.  In fact,  similar to a L\'evy process with an infinite L\'evy measure,  the $Q$-process in this case includes numerous ``small" non-local jumps near $\infty$ (accordingly, a $Q$-process with finite $\nu$ only allows ``big" non-local jumps).  These ``small" non-local jumps make the $Q$-process resemble a ``diffusion" near $\infty$ rather than a jump process.  Although the heuristic construction above is not applicable to such a $Q$-process,  it can be approximated by a specific sequence of $Q$-processes with only ``big" non-local jumps (see Theorem~\ref{THM92}).  Our approximating strategy in this result is only valid when $\infty$ is regular. We intend to further address this issue for general cases in a future {paper}.


The approach involving Ray-Knight compactification is simpler and more comprehensible than the two approaches previously introduced in \cite{F59} and \cite{WY92}.  It is especially accessible to those familiar with the general theory of Markov processes.  {However, most proofs in this treatment still depend on the resolvent representation of $Q$-processes derived through the analytic approach in \cite{WY92}.}

Finally, let us briefly examine symmetrizable $Q$-processes.
A $Q$-process with transition matrix $(p_{ij}(t))$ is called \emph{symmetrizable} if there exists a $\sigma$-finite measure $\mu$ on $\bN$ such that $\mu_ip_{ij}(t)=\mu_j p_{ji}(t)$ for all $i,j\in \bN$ and $t\geq 0$,  where $\mu_i:=\mu(\{i\})$.  A {symmetrizing} measure of $(p_{ij}(t))$ is always a {symmetrizing} measure of $Q$,  that is $\mu_i q_{ij}=\mu_j q_{ji}$ for all $i,j\in \bN$; see \cite[Lemma~6.6]{C04}.  Hence, $\mu$ is unique up to a multiplicative constant.  In this paper, we always consider $\mu$ as
\begin{equation}\label{eq:03}
\mu_0=1,  \quad \mu_k=\frac{b_0b_1\cdots b_{k-1}}{a_1a_2\cdots a_k}, \; k\geq 1.  
\end{equation}
This measure is also known as the \emph{speed measure} of the (minimal) $Q$-process. 
Typical examples of symmetrizable $Q$-processes are the minimal one and the $(Q,1)$-process.  

When dealing with symmetrizable $Q$-processes, we {will use} Dirichlet forms.  For terminologies and notations concerning Dirichlet forms, we refer {the readers} to \cite{C04}, \cite{CF12}, and \cite{FOT11}.
A symmetrizable $Q$-process is always associated with a Dirichlet form.   Particularly, the minimal $Q$-process corresponds to the \emph{minimal Dirichlet form},  while $(Q,1)$-process corresponds to the \emph{basic Dirichlet form}; see \cite[\S6]{C04} and also Section~\ref{SEC2}.  With the help of the theory of Dirichlet forms,  we can further show that other symmetrizable $Q$-processes can be described as the killed processes of the $(Q,1)$-process parameterized by a constant $\kappa\in (0,\infty)$; see Theorem~\ref{THM73}. 




The remaining sections of this paper are organized as follows. In Section \ref{heuristics}, we will provide a summary of the main results, along with a {heuristic explanation of the main ideas}. Section \ref{SEC2} serves as an introduction to two important $Q$-processes: the minimal $Q$-process and the $(Q,1)$-process. These processes are closely tied to the uniqueness of $Q$-process. {We construct them using time-changed Brownian motion}, a method that is particularly convenient for {those familiar} with Dirichlet forms. Section \ref{SEC3} establishes the Ray-Knight compactification of a $Q$-process. Section \ref{SEC4} is dedicated to the case where $\infty$ is branching, which corresponds to the $Q$-process being a Doob process. The following sections will examine the other case, where $\infty$ is non-branching, resulting in the $Q$-process being a Feller process.


\subsection*{Notations}
{The following parameters will be frequently used}:
\[
	c_0=0,  \quad c_1=\frac{1}{2b_0}, \quad c_k=\frac{1}{2b_0}+\sum_{i=2}^k \frac{a_1a_2\cdots a_{i-1}}{2b_0b_1\cdots b_{i-1}}, \;k\geq 2.  
\]
This sequence is also known as the \emph{scale function} of the (minimal) $Q$-process. 

Let $\overline{\bR}:=[-\infty, \infty]$ denote the extended real number system.   
The Alexandroff compactification of $\bN$ is represented by $\bN\cup \{\infty\}$.  In other words,  $\bN\cup \{\infty\}$ equipped with a suitable metric becomes a compact metric space.  The trap point $\partial$ for a $Q$-process is typically treated as an isolated point that is separate from $\bN$,  except in special cases (see Remark~\ref{RM32} and the second paragraph in Section~\ref{SEC4}).  

For a positive measure $\nu$ on $\bN$,  let $\nu_k$ denote $\nu(\{k\})$ and let $|\nu|$ represent $\sum_{k\in \bN}\nu_k$.  Given an appropriate function $f$ on $\bN$,  $\nu(f)$ stands for the integration of $f$ with respect to $\nu$.  Regarding the measure $\mu$ in \eqref{eq:03},  if necessary, we always assume that $\mu(\{\infty\})=0$, thereby extending it to $\bN\cup \{\infty\}$.


In a strict sense,  a right process $X$ on a state space $E$ should be defined as a six-tuple $X=\{\Omega, \sF,  \sF_t, X_t, \theta_t, \mathbf{P}_x\}$,  where $(\Omega, \sF)$ represents the sample space,  $\sF_t$ is the filtration on the sample space,  $\theta_t$ represents the translation operator, and for every $x\in E$,  $\mathbf{P}_x$ is the probability measure on $(\Omega,\sF)$ that satisfies $\bP_x(X_0=x)=1$.  The trap point for $X$ is also denoted by $\partial$ if there is no danger of confusion,  and we also say $X$ is a right process on $E\cup \{\partial\}$.   Additional details regarding these notations can be found in \cite{S88}.  It is important to note that if $X$ is a Ray process instead of a right process,  most notations remain the same, but $\bP_x(X_0=x)=1$ only holds for non-branching points $x$.


\section{Main results and heuristics}\label{heuristics}

\subsection{$Q$-processes and Doob's separable modification}


Let $X=(X_t)_{t\geq 0}$ be  a $Q$-process on the sample space $(\Omega, \sF, \mathbf{P})$ with standard transition matrix $(p_{ij}(t))_{i,j\in \bN}$ and initial distribution being a probability measure on $\bN$.  This transition matrix is treated as an honest one on $\bN\cup \{\partial\}$,  where $\partial$ is the trap point.  
For every sample $\omega\in \Omega$,  $X_\cdot(\omega)$ is a $\overline{\bR}\cup \{\partial\}$-valued function,  and $\zeta(\omega):=\inf\{t\geq 0: X_t(\omega)=\partial\}$ is called the \emph{lifetime} of $X$.  {The index set $\bN$ is known as the \emph{minimal state space}.  Another important concept of the ``state space" is the so-called \emph{essential range},  defined to be the set of values $x\in \overline{\mathbb{R}}$ that are contained in the range of $X_\cdot(\omega)$ for a set of $\omega$ of positive probability.   A value in the essential range will be called a \emph{state} of $X$,  and it will be called a \emph{fictitious state} if and only if it is not in $\mathbb{N}$.  According to \cite[II \S4, Theorem~2]{C67},  the only possible fictitious state of $X$ is $\infty$.  (This is because,  $\bN$ is equipped with the discrete topology.)} 



One of the most crucial tools in the study of continuous-time Markov chains is Doob’s \emph{separable modification}.  This technique allows for the creation of a ``canonical version" of the $Q$-process that possesses the properties of being well-separable {(for its definition, see page 138 of \cite{C67})}, Borel measurable, and right lower semicontinuous; see \cite[II \S7, Theorem~1]{C67}. 


\subsection{Ray-Knight compactification of $Q$-processes}\label{SEC22}

Our approach to the Ray-Knight compactification of $X$ begins with Doob's canonical version.  However, for this approach in Theorem \ref{THM32},  only the Borel measurability of $X$ is required.  According to Lemma~\ref{LM33}~(1), the resulting Ray process $\bar{X}$ is a c\`adl\`ag modification of the given $Q$-process. The state space of $\bar{X}$ is $\bN\cup\{\infty\}\cup \{\partial\}$, where $\partial$ remains the trap in the context of Markov processes, regardless of whether $\infty$ is a fictitious state for $X$ or not.  As mentioned in Section~\ref{SEC1},  we call $\bar{X}$ a Ray $Q$-process.  Depending on whether $\infty$ is branching or not, it is classified into two categories: Doob processes and Feller $Q$-processes.



{Doob processes never arrive at $\infty$, but Feller $Q$-processes always do.} 
In the context of general Markov processes, the initial distribution of a Feller $Q$-process can be a probability measure on $\bN\cup \{\infty\}$, rather than just on $\bN$.


Doob processes are not Feller processes on $\mathbb{N}\cup\{\infty\}\cup \{\partial\}$ because they do not satisfy the quasi-left-continuity at the first flying time.  However, a specific Doob process,  the minimal $Q$-process,  can be considered as a Feller process on $\mathbb{N}\cup \{\partial\}$ if we identify the trap $\partial$ with $\infty$ (namely,  $\mathbb{N}\cup \{\partial\}$ is perceived as the Alexandroff compactification of $\mathbb{N}$).



\subsection{Boundary behaviours of Feller $Q$-processes}\label{SEC23-2}

The resolvent of a Feller $Q$-process can be characterized by a triple $(\nu, \gamma,  \beta)$, {where $\nu$ is a positive measure on $\bN$ and $\gamma$ and $\beta$ are two non-negative constants,  as in Lemma~\ref{LM62}}. By utilizing this resolvent representation,  the infinitesimal generator of the process is formulated in Theorem~\ref{THM62}.  It is a sub-operator of the \emph{maximal (discrete) generalized second order differential operator} {satisfying} the following boundary condition: for any $F$ in the domain of the infinitesimal generator, 
\begin{equation}\label{eq:21-2}
{\frac{\beta}{2}} F^+(\infty)+\sum_{k\geq 0}(F(\infty)-F(k))\nu_k +\gamma F(\infty)=0,
\end{equation}
where $F^+$ is the (discrete) first order derivative of $F$ as defined in \eqref{eq:61-2}. 


{We first give some special cases with $|\nu|=0$. We note in passing that case (a) does not satisfy all the conditions in Lemma~\ref{LM62}.}
\begin{itemize}
\item[(a)] When $|\nu|=0$ and $\gamma\neq 0, \beta=0$,  \eqref{eq:21-2} represents the \emph{Dirichlet boundary condition}.  In this case, the $Q$-process reduces to the minimal one.  (In the setting of Appendix~\ref{APPB},  $\gamma=0$ with $|\nu|=\beta=0$ also corresponds to the minimal $Q$-process.) 
\item[(b)] When $|\nu|=0$ and $\gamma=0, \beta\neq 0$,  \eqref{eq:21-2}  represents the \emph{Neumann boundary condition}.  The corresponding $Q$-process is known as the $(Q,1)$-process, but it could be better called the \emph{reflecting $Q$-process}.  It is worth noting that this case occurs only when $\infty$ is regular. 
\item[(c)] When $|\nu|=0$ and $\gamma,\beta\neq 0$,  \eqref{eq:21-2} represents the \emph{Robin boundary condition}, as established in \cite[Theorem~11.1]{F59}.  The corresponding $Q$-process is known as the \emph{elastic $Q$-process}.  This situation also occurs only when $\infty$ is regular.
\end{itemize}
These boundary conditions, in the context of Laplace operators, correspond to minimal (or absorbing) Brownian motion, reflecting Brownian motion, and elastic Brownian motion, respectively.

Let us provide a probabilistic interpretation of \eqref{eq:21-2} for the special cases outlined above.  The cases (a) and (b) are clear and require no further explanation. We now focus on the Robin boundary condition (c). The corresponding elastic $Q$-process $\bar{X}$ is a symmetrizable $Q$-process with the parameter $\kappa=\gamma/\beta$, as stated in Theorem~\ref{THM73}.  It can be described as the killed process of the $(Q,1)$-process at the killing time
\begin{equation}\label{eq:13}
	\zeta=\inf\{t>0: L_t>\xi\},
\end{equation}
where $(L_t)_{t\geq 0}$ represents the local time of the $(Q,1)$-process at $\infty$,  as detailed in,  e.g., \cite[Theorem~A.2.11]{FOT11}. In other words,  $(L_t)_{t\geq 0}$ is the positive continuous additive functional (PCAF for abbreviation) associated with the Dirac measure at $\infty$. Additionally,  $\xi$ is {an exponential random variable with mean} $\beta/\gamma$,  independent of the $(Q,1)$-process. The function $L_t$ in $t$ only increases when the $(Q,1)$-process visits $\infty$,  {similar to} the behaviour of the local time for reflecting Brownian motion. Therefore, $L_t$ can be used to quantify the intensity of reflection at $\infty$ before time $t$. Equation \eqref{eq:13} signifies that $\xi$ captures the maximum intensity of reflection for $\bar{X}$ at $\infty$ before it is killed. As the mean of $\xi$ is $\beta/\gamma$, we {see} that larger values of $\beta$ result in longer durations of reflection for $\bar{X}$ at $\infty$, while larger values of $\gamma$ lead to earlier termination. We can thus interpret $\beta$ as the mean reflecting intensity and $\gamma$ as the mean killing intensity for $\bar{X}$, offering valuable insights into the boundary behaviours of the process.


Regarding the general boundary condition \eqref{eq:21-2}, we adopt the convention $F(\partial)=0$ and define $\tilde{\nu}|_{\bN}:=\nu,\tilde{\nu}({\partial}):=\gamma$ and $|\tilde{\nu}|:=|\nu|+\gamma$.  Then \eqref{eq:21-2} can be written as
\begin{equation}\label{eq:22-2}
{\frac{\beta}{2}} F^+(\infty)+\sum_{k\in \bN\cup \{\partial\}}(F(\infty)-F(k))\tilde \nu_k =0.
\end{equation}
{The corresponding $Q$-process $\bar X$ can be constructed from the $(Q,1)$-process in the following informal way: Let $\xi$ be {an exponential random variable with mean} $\beta/|\tilde{\nu}|$  similar to that in \eqref{eq:13}, but with the corresponding $\zeta$ being the time of the first non-local jump, rather than the killing time. When this jump occurs, the position that $\bar X$ reaches is determined by the ``distribution" $\tilde{\nu}/|\tilde{\nu}|$.  Repeat this operation until $\bar X$ enters the trap $\partial$. In this context, the parameter $\beta$ {still plays} its role of reflecting intensity. However, the parameter $\tilde{\nu}$, as opposed to $\gamma$ in case (c), possesses a more intricate and profound significance.}

This probabilistic interpretation for the general boundary condition \eqref{eq:22-2} is made rigorous in Theorem~\ref{THM61} under the condition $|\nu|<\infty$. However, when $|\nu|=\infty$, the construction appears logically inconsistent as $\beta/|\tilde{\nu}|=\tilde{\nu}_k/|\tilde{\nu}|=0$ for all $k\in \bN\cup \{\partial\}$.  {For this ``pathological" case,  although we have not found any pathwise construction for birth-and-death processes, it can be explained that $\bar{X}$ can indeed return to $\mathbb{N}$ from $\infty$ through non-local jumps, even though these non-local jumps may not be easily observable.}
To convince the readers,  consider the special case where $\infty$ is an exit  (thus $\beta=0$) with $|\nu|=\infty$ and $\gamma=0$. Let
\[
	\tau^1_0:=\inf\{t>\eta: \bar X_t=0\}
\]
represent the first hitting time of $0$ after the first flying time $\eta$ as defined in \eqref{eq:30} for $\bar{X}$. We further define
\[
\tau:=\sup\{t: t\leq \tau^1_0,  t\in L(\omega)\},
\]	
where $L(\omega)$ is the set of flying times as defined before Corollary~\ref{COR35}. In other words, $\tau$ represents the last flying time before $\tau^1_0$.  Due to the c\`adl\`ag property of $\bar X$ and Corollary~\ref{COR35}, it is evident that $\bar X_{\tau-}=\infty$.  Additionally, an important finding in \cite[\S6.5, Theorem~2]{WY92} reveals that $\mathbf{P}(\bar X_\tau=\infty)=0$, and if $\nu_j>0$ for $j\in \bN$, then $\mathbf{P}(\bar X_\tau=j)>0$. 


The boundary condition \eqref{eq:21-2} also provides guidelines for symmetrization of $Q$-processes.  Informally stated,  a symmetrizable $Q$-process remains the same after being reversed in time.  Therefore, it cannot exhibit non-local jumps from $\infty$ to $\mathbb{N}$,  as no $Q$-processes allow for non-local jumps from $\mathbb{N}$ to $\infty$.  This explains why the minimal $Q$-process, the $(Q,1)$-process, and the elastic $Q$-processes mentioned in (a), (b), and (c) are all the symmetrizable $Q$-processes, as asserted in Theorem~\ref{THM73}.


 \subsection{Comparison with diffusions on an interval}
 
 As suggested by Feller \cite{F59}, a $Q$-process can be considered as a “diffusion process”. A typical diffusion process $Y=(Y_t)_{t\geq 0}$ on a closed interval $[l,r]$ is {described by} a \emph{scale function} $s$ and a \emph{speed measure} $m$. The scale function $s$ is a continuous and strictly increasing function on the interval $(l,r)$, while the speed measure $m$ is a fully supported Radon measure on $(l,r)$. The pair $(s,m)$, which characterizes the behaviour of $Y$ within the interior of $[l,r]$, plays a role similar to that of $Q$ in birth and death processes.

In the case of a Feller diffusion process on $[l,r]$,  every function in the domain of its infinitesimal generator satisfies an analogous boundary condition at a regular or exit boundary $r$ (or $l$) to \eqref{eq:21-2}. However, an additional \emph{sojourn term} may appear in this boundary condition, as discussed in \cite[II \S5, Theorem~2]{M68}.  The non-trivial sojourn at $r$ indicates that the set $\{t: Y_t=r\}$ has a positive Lebesgue measure, suggesting that the diffusion process spends a significant amount of time at the boundary. Fortunately, when dealing with $Q$-processes, we do not encounter this situation, as the total visiting time at $\infty$ for a $Q$-process must be zero.
 
 
 Nevertheless, a \emph{generalized $Q$-process} can still be considered, allowing for non-trivial sojourn at $\infty$ (see, e.g., \cite[\S12-15]{F59}). In this {case}, it is necessary to include $\infty$ in the minimal state space. Additionally,  the real birth-death matrix should be an extension of $Q$ to $\mathbb{N}\cup \{\infty\}$. In the context of the Kolmogorov backward equation \eqref{eq:02}, this generalized $Q$-process corresponds to a solution $(p_{ij}(t))_{i,j\in \mathbb{N}}$ that satisfies the inequality
\[
	\sum_{k\in \bN}p_{ik}(t)p_{kj}(t)\leq p_{ij}(t+s),\quad i,j\in \bN,  t,s\geq 0,
\]
 instead of the Chapman-Kolmogorov equation.
 

\subsection{Comparison with L\'evy processes}

Non-local jumps of a $Q$-process with $|\nu|=\infty$ can also be compared with L\'evy jumps.  Let us imagine $\mathbb{N}\cup \{\infty\}$ as an interval $[l,r]$ where $r$ corresponds to $\infty$. Similarly, let $(c, \mu)$ be imagined as the pair of scale function $s$ and speed measure $m$ on $[l,r]$. For simplification, we assume $r=0$. Then the first identity in \eqref{eq:B1} can be written as
\begin{equation}\label{eq:24-2}
\int_{[l,0)} \left( \int_x^0 m([l,y])ds(y)\right) \hat{\nu}(dx)<\infty,\end{equation}
where $\hat{\nu}$ is the analogue of $\nu$ on $[l,0]$. In the Brownian case, where $s(x)=x$ and $m$ is the Lebesgue measure, \eqref{eq:24-2} amounts to
\[\hat{\nu}([l, e])<\infty,\quad 	\int_{[e,0)} |x|\hat{\nu}(dx)<\infty,
\]
where $e$ is an arbitrary fixed point in $(l,0)$. By this observation, $\nu$ has an analogous form to the L\'evy measure of a \emph{subordinator}.




\section{Uniqueness criteria}\label{SEC2}

\subsection{Minimal $Q$-process as a quasidiffusion}\label{SEC21}

Let $\widehat{\bN}:=\{c_k: k\geq 0\}$ and 
\[
\Xi: \widehat{\bN}\rightarrow \bN,\quad c_k\mapsto k.  
\]
Denote by $\widehat{\mu}:=\mu\circ \Xi$ the positive measure on $\widehat{\bN}$ with $\widehat{\mu}(\{c_k\}):=\mu_k$ for $k\geq 0$.  Set $c_\infty:=\lim_{k\uparrow \infty} c_k\leq \infty$.  

We provide a probabilistic construction of the minimal $Q$-process using a time change transformation. 
Firstly,  let us consider a Brownian motion $W^\text{min}:=(W^\text{min}_t)_{t\geq 0}$ with lifetime $\zeta^{W^\text{min}}$ on the interval $[c_0, c_\infty)$,  which is reflected at $c_0$ and absorbed at $c_\infty$ {if} $c_\infty<\infty$.  In other words,  $W^\text{min}$ is symmetric with respect to the Lebesgue measure and associated with the regular Dirichlet form $(\frac{1}{2}\mathbf{D}, H^1_0([c_0,c_\infty))$ on $L^2([c_0,c_\infty))$: 
\[
\begin{aligned}
&H^1_0([c_0,c_\infty))=\{f\in L^2([c_0,c_\infty)): f\text{ is absolutely continuous on }(c_0,c_\infty) \\
 &\qquad\qquad \qquad \qquad\qquad\qquad \text{and } f'\in L^2([c_0,c_\infty)),  f(c_\infty)=0\text{ if }c_\infty<\infty\},\\
&\frac{1}{2}\mathbf{D}(f,g)=\frac{1}{2}\int_{c_0}^{c_\infty} f'(x)g'(x)dx,\quad f,g\in H^1_0([c_0,c_\infty)).   
\end{aligned}
\]
Next,  we use $\widehat{\mu}$ to obtain a time-changed Brownian motion $\widehat W^\text{min}=(\widehat W^\text{min}_t)_{t\geq 0}$ on $\widehat{\bN}$.  More precisely,  $\widehat{\mu}$,  viewed as a measure on $[c_0,c_\infty)$ by setting $\widehat{\mu}([c_0,c_\infty)\setminus \widehat{\bN})=0$,  is clearly a Radon smooth measure with respect to $W^\text{min}$,  and hence induces a PCAF $({A}^\text{min}_t)_{t\geq 0}$.  Note that ${A}^\text{min}_t=\sum_{k\geq 0} \mu_k  L^\text{min}_t({c_k})$,  where ${L}^\text{min}_t({c_k})$ is the local time of $W^\text{min}$ at $c_k$.  We set 
\[
\tau^\text{min}_t:=\left\lbrace 
	\begin{aligned}
		&\inf\{s>0: {A}^\text{min}_s>t\},\quad t<{A}^\text{min}_{{\zeta}^{W^\text{min}}-},\\
		&\infty,\qquad\qquad\qquad\;\; t\geq {A}^\text{min}_{{\zeta}^{W^\text{min}}-}. 
	\end{aligned}
\right.
\]
Then,  $\widehat W^\text{min}_t:=W^\text{min}_{\tau^\text{min}_t}, t\geq 0$, define a time-changed Brownian motion on $\widehat{\bN}$  with the lifetime $\widehat{\zeta}^\text{min}:={A}_{{\zeta}^{W^\text{min}}-}$.  It is also called a \emph{quasidiffusion} in,  e.g.,  \cite{BK87,  K86},  a \emph{generalized diffusion} in,  e.g.,  \cite{W74,  KW82} and a \emph{gap diffusion} in,  e.g.,  \cite{K81}. The measure $\widehat{\mu}$ is called the speed measure of $\widehat W^\text{min}$.  Finally, we define
\begin{equation}\label{eq:21}
	X^\text{min}_t:=\Xi(\widehat W^\text{min}_t),\quad t\geq 0.  
\end{equation}
The lifetime of $X^\text{min}$ is denoted by $\zeta^\text{min}$.  
This process is actually the desired minimal $Q$-process.

\begin{lemma}\label{LM21}
The process $X^\text{min}=(X^\text{min}_t)_{t \geq 0}$,  as defined in \eqref{eq:21}, is {equal to} the minimal $Q$-process.  
\end{lemma}
\begin{proof}
As shown in \cite[Theorems 3.4 and 3.6]{L22},  $\widehat{W}^\text{min}$ is symmetric with respect to $\widehat{\mu}$ and is associated with the regular Dirichlet form on $L^2(\widehat{\bN},\widehat{\mu})$:
\[
\begin{aligned}
&\widehat{\sG}^\text{min}:=\left\{\widehat{f}\in L^2(\widehat{\bN},\widehat{\mu}): \widehat{\sA}^\text{min}(\widehat{f},\widehat{f})<\infty,  \widehat{f}(c_\infty):=\lim_{k\uparrow \infty} \widehat{f}(c_k)=0\text{ if }c_\infty<\infty\right\}, \\
&\widehat{\sA}^\text{min}(\widehat{f},\widehat{g})=\frac{1}{2}\sum_{k\geq 0}\frac{(\widehat{f}(c_{k+1})-\widehat{f}(c_k))(\widehat{g}(c_{k+1})-\widehat{g}(c_k))}{|c_{k+1}-c_k|},\quad \widehat{f},\widehat{g}\in \widehat{\sG}^\text{min}.  
\end{aligned}\]
Hence, $X^\text{min}$ is symmetric with respect to $\mu$ and is associated with the regular Dirichlet form on $L^2(\bN,\mu)$:
\[
\begin{aligned}
&{\sG}^\text{min}:=\left\{{f}\in L^2({\bN},{\mu}): {\sA}^\text{min}({f},{f})<\infty,  {f}(\infty):=\lim_{k\uparrow \infty}f(k)=0\text{ if }c_\infty<\infty\right\}, \\
&{\sA}^\text{min}({f},{g})=\frac{1}{2}\sum_{k\geq 0}\frac{({f}({k+1})-{f}(k))({g}({k+1})-{g}(k))}{|c_{k+1}-c_k|},\quad {f},{g}\in {\sG}^\text{min}.  
\end{aligned}\]
Applying \cite[Proposition~6.59]{C04},  we can conclude that $X^\text{min}$ is {equal to} the minimal $Q$-process.  
\end{proof}
\begin{remark}\label{RM32}
{In this proof, the trap point $\partial$ for $X^\text{min}$ is assumed to be the Alexandroff compactification point of $\bN$, that is, $\partial=\infty$.  With this setting, $X^\text{min}$ is a Hunt process.  When $\partial\neq \infty$ is considered as an isolated point {attached} to $\bN$,   we may put $(\sA^\text{min},\sG^\text{min})$ on $L^2(\bN\cup\{\infty\},  \mu)$,  where $\mu$ extends to $\bN\cup\{\infty\}$ by $\mu(\{\infty\})=0$.  It is quasi-regular rather than regular.  Accordingly, $X^\text{min}$ can be viewed as a Borel standard process on $\bN\cup\{\infty\}$, and $\{\infty\}$ is an $\sA^\text{min}$-polar set. }
\end{remark}

The Dirichlet form $(\sA^\text{min},\sG^\text{min})$ in this proof is referred to as the \emph{minimal Dirichlet form} in \cite{C04}.  Another important Dirichlet form,  known as the \emph{basic Dirichlet form},  is defined as follows:
\[
\begin{aligned}
&{\sG}^*:=\left\{{f}\in L^2({\bN},{\mu}): {\sA}^*({f},{f})<\infty\right\}, \\
&{\sA}^*({f},{g})=\frac{1}{2}\sum_{k\geq 0}\frac{({f}({k+1})-{f}(k))({g}({k+1})-{g}(k))}{|c_{k+1}-c_k|},\quad {f},{g}\in {\sG}^*.  
\end{aligned}
\]
It corresponds to another symmetrizable $Q$-process,  namely the $(Q,1)$-process,  {when $\infty$ is a regular boundary as defined in Definition~\ref{DEF23}}; see,  e.g.,  \cite[Theorem~6.55]{C04}.  

\subsection{Uniqueness criteria}

Set two parameters
\[
	\sigma:=\int_{c_0}^{c_\infty}\widehat{\mu}([0,x])dx=\sum_{k=0}^\infty (c_{k+1}-c_k)\cdot \sum_{i=0}^k \mu_i, \quad \lambda:=\int_{c_0}^{c_\infty} x\widehat{\mu}(dx)=\sum_{k\geq 0} c_k \mu_k. 
\]
The following classification for the boundary point $\infty$ in Feller's sense is well known.

\begin{definition}\label{DEF23}
The boundary point $\infty$ (for $X^\text{min}$) is called
\begin{itemize}
\item[(1)] \emph{regular},  if ${\sigma}<\infty,  {\lambda}<\infty$;
\item[(2)] \emph{an exit},  if ${\sigma}<\infty,  {\lambda}=\infty$;
\item[(3)] \emph{{an entrance}},  if ${\sigma}=\infty,  {\lambda}<\infty$;
\item[(4)] \emph{natural},  if ${\sigma}= {\lambda}=\infty$.
\end{itemize}
\end{definition}
\begin{remark}
Note that $\infty$ is regular if and only if $c_\infty+\mu(\bN)<\infty$.   If $\infty$ is an exit,  then $c_\infty<\infty$ and $\mu(\bN)=\infty$.  If $\infty$ is {an entrance},  then $c_\infty=\infty$ and $\mu(\bN)<\infty$.  If $\infty$ is natural,  then $c_\infty+\mu(\bN)=\infty$. 
\end{remark}

It is also well known that {the transition matrix of $X^\text{min}$ is honest},  if and only if $\sigma=\infty$ (namely,  $\infty$ is {an entrance} or natural); see,  e.g.,  \cite[Proposition~7.9]{L22}.  As a result,  $\sigma=\infty$ is equivalent to the uniqueness of $Q$-processes.  In the case where $\infty$ is regular or an exit,  another uniqueness criterion has been proposed for symmetrizable $Q$-processes; see, e.g., \cite[Corollary~6.62]{C04}.  More precisely,  the symmetrizable {$Q$-process is unique} if and only if the basic Dirichlet form is {equal to} the minimal one,  i.e.,  $(\sA^*,\sG^*)=(\sA^\text{min},\sG^\text{min})$.  These uniqueness criteria are summarized in the following result.

\begin{theorem}\label{THM25}
\begin{itemize}
\item[\rm (1)] The symmetrizable {$Q$-process is unique} if and only if $\infty$ is not regular,  meaning that $c_\infty=\infty$ or $\mu(\bN)=\infty$.  
\item[\rm (2)] The uniqueness of $Q$-processes holds if and only if $\infty$ is {an entrance} or natural,  i.e. 
\[
\sum_{k=1}^\infty (c_{k+1}-c_k)\cdot \sum_{i=0}^k \mu_i =\infty.
\]
\end{itemize}
\end{theorem}
\begin{proof}
The second assertion is clear.  We only prove the first one. When $\infty$ is regular,  we will obtain another symmetrizable $Q$-process in Section~\ref{SEC23}, which demonstrates that the uniqueness fails. When $\infty$ is an exit,  $c_\infty<\infty$ and thus $f(c_\infty)$ exists for any $f\in \sG^*$.  However,  we have $\mu(\bN)=\infty$.  It follows that $f(c_\infty)=0$ because $f\in L^2(\bN,\mu)$.  Therefore,  $(\sA^*,\sG^*)=(\sA^\text{min},\sG^\text{min})$ and the first assertion can be concluded.  This completes the proof. 
\end{proof}

\subsection{$(Q,1)$-process}\label{SEC23}

In this subsection, we examine the case where $\infty$ is regular and present a probabilistic construction for another symmetrizable $Q$-process,  namely the $(Q,1)$-process.  This construction is analogous to the one used to obtain the minimal $Q$-process in Section~\ref{SEC21}.   

Note that $c_\infty<\infty$ and $\widehat{\mu}(\widehat{\bN})=\mu(\bN)<\infty$.  
We take a reflecting Brownian motion $W^*:=(W^*_t)_{t\geq 0}$ on the closed interval $[c_0,c_\infty]$ and extend  the measure $\widehat{\mu}$ to $[c_0,c_\infty]$ by letting $\widehat{\mu}([c_0,c_\infty]\setminus \widehat{\bN}):=0$.  Note that $\widehat{\mu}$ is a Radon smooth measure with respect to $W^*$ and corresponds to the PCAF ${A}^*_t=\sum_{k\geq 0}\mu_k L^*_t(c_k)$,  where $L^*_t(c_k)$ represents the local time of $W^*$ at $c_k$.  Set $\tau^*_t:=\inf\{s>0: A^*_s>t\}$ for any $t\geq 0$ and then define a time-changed Brownian motion on $\widehat{\bN}\cup \{c_\infty\}$: 
\[
\widehat{W}^*_t:=W^*_{\tau^*_t},\quad t\geq 0. 
\]
Finally let
\begin{equation}\label{eq:22}
	X^*_t:= \Xi(\widehat{W}^*_t),\quad t\geq 0.  
\end{equation}
It is easy to verify that $X^*=(X^*_t)_{t\geq 0}$ is a $\mu$-symmetric Markov process on $\bN\cup \{\infty\}$.  Denote by $\mathbf{P}^*_i$ the probability measure in the six-tuple representation of $X^*$ for all $i\in \bN\cup \{\infty\}$.   

\begin{proposition}\label{PRO23}
Let $X^*$ be defined as \eqref{eq:22}.  Then the Dirichlet form of $X^*$ on $L^2(\bN\cup\{\infty\},\mu)$ is equal to the basic Dirichlet form $(\sA^*,\sG^*)$,  and
\begin{equation}\label{eq:33-2}
	p^*_{ij}(t):=\mathbf{P}^*_i(X^*_t=j),\quad i,j\in \bN, t\geq 0
\end{equation}
is a standard,  honest, symmetrizable transition matrix with density matrix $Q$.  
\end{proposition}
\begin{proof}
By applying \cite[Theorems~3.4 and 3.6]{L22},  it can be shown that the Dirichlet form of $X^*$ on $L^2(\bN\cup\{\infty\}, \mu)$ is exactly the basic Dirichlet form $(\sA^*,\sG^*)$.  
Furthermore,  it should be noted that the time-changed Brownian motion $\widehat{W}^*$ is also a quasidiffusion,  as demonstrated in \cite{L22}. 
Thus,  the transition semigroup of $\widehat{W}^*$ is absolutely continuous with respect to $\widehat{\mu}$; see,  e.g.,  \cite{K75}.  It follows that $P^*_t(x,\cdot)=\mathbf{P}^*_x(X^*_t\in \cdot )$ is absolutely continuous with respect to $\mu$.  Particularly,  $P^*_t(x, \{\infty\})=0$ for any $x\in \bN\cup\{\infty\}$.  Hence,  it can be straightforwardly verified that $(p^*_{ij}(t))_{i,j\in \bN}$ is an honest and symmetrizable transition matrix. 

We need to prove that $(p^*_{ij}(t))$ is a $Q$-process.  Clearly, it is standard because the sample paths of $X^*$ are right continuous.  We now {show}
\begin{equation}\label{eq:23}
	\lim_{t\downarrow 0}\frac{p^*_{ij}(t)-\delta_{ij}}{t}=q_{ij},\quad i,j\in \bN.  
\end{equation}
Set $\eta_\infty:=\inf\{t>0: X^*_t=\infty\}$.  The stopped process of $X^*$ at $\eta_\infty$ is exactly the minimal $Q$-process $X^\text{min}$.  Consequently,  we have
\[
	\lim_{t\downarrow 0}\frac{\bP^*_i(X^*_t=j,  t<\eta_\infty)-\delta_{ij}}{t}=q_{ij}, \quad i,j\in \bN.
\]
To establish \eqref{eq:23},  it suffices to prove that
\[
	\lim_{t\downarrow 0} \frac{\bP^*_i(\eta_\infty\leq t)}{t}=0.  
\]
This proof follows a standard approach by utilizing the strong Markov property and the ``local" property of $X^*$ (see, e.g.,  \cite[\S5.2]{I06}).  We include the details for the convenience of {the readers}.

We will proceed by contradiction and assume that
\[
\limsup_{t\downarrow 0}\frac{\bP^*_i(\eta_\infty\leq t)}{t}>0.  
\]
Therefore,  there is a constant $c>0$ and a decreasing sequence $t_n\downarrow 0$ such that 
\begin{equation}\label{eq:25}
	\bP^*_i(\eta_\infty\leq t_n)\geq ct_n.  
\end{equation}
Let us define $Y_t:=X^*_t$ for $t<\eta_\infty$ and $Y_t:=\infty$ for $t\geq \eta_\infty$.  Now,  we will  proceed with several steps to arrive at a contradiction.

Firstly,  define $\Lambda_s:=\{\omega\in \Omega: Y_t(\omega)=i, 0\leq t\leq s\}$ for any $s>0$.  It is clear that $\Lambda_s$ is increasing as $s\downarrow 0$ and $\Omega\subset \cup_{s>0}\Lambda_s$,  $\bP^*_i$-a.s.  Hence, there exists $s>0$ such that
\begin{equation}\label{eq:24}
	\bP^*_i(\Lambda_s)>1/2.  
\end{equation}
Keeping this $s>0$ fixed, we proceed by defining,   for any $n\geq 1$,
\[
K_n:=\cup_{\{m\geq 1: mt_n\leq s\}}\{Y_{(m-1)t_n}=i,  Y_{mt_n}=\infty\}.
\]
It is evident that
\[
	K_n\subset \tilde{K}_n:=\{\omega\in \Omega: \exists r_1<r_2\leq s, \text{ s.t. }r_2-r_1\leq t_n,  Y_{r_1}(\omega)=i,  Y_{r_2}(\omega)=\infty\},
\]
and $\tilde{K}_n$ is decreasing as $n\uparrow \infty$.  Moreover, we have $\bP^*_i(\cap_{n\geq 1}\tilde{K}_n)=0$.  Consequently,
\begin{equation}\label{eq:26}
\bP^*_i(K_n)\rightarrow 0,\quad n\rightarrow \infty.
\end{equation}
Finally,  let us set 
\[
	K_{n,m}:=\{Y_{(m-1)t_n}=i,  Y_{mt_n}=\infty\}.
\]
In other word, $K_n=\cup_{\{m\geq 1: mt_n\leq s\}}K_{n,m}$.  Note that $K_{n,m}\cap K_{n,m'}=\emptyset$ for $m\neq m'$ and 
\[
	K_{n,m}\supset K'_{n,m}:=\{Y_t=i,  \forall 0\leq t\leq (m-1)t_n\text{ and }(m-1)t_n<\eta_\infty\leq mt_n\}.
\]
The strong Markov property of $X^*$ implies that
\[
	\bP^*_i(K'_{n,m})=\bP^*_i(\eta_\infty\leq t_n)\bP^*_i(Y_t=i, \forall 0\leq t\leq (m-1)t_n). 
\]
Using \eqref{eq:25},  \eqref{eq:24} and $(m-1)t_n\leq s$,  we have $\bP^*_i(K'_{n,m})\geq ct_n/2$.  Therefore,  it follows that
\[
\begin{aligned}
\bP^*_i(K_n)&=\sum_{\{m\geq 1: mt_n\leq s\}}\bP^*_i(K_{n,m})\geq \sum_{\{m\geq 1: mt_n\leq s\}}\bP^*_i(K'_{n,m}) \\
 &\geq  \sum_{\{m\geq 1: mt_n\leq s\}} \frac{ct_n}{2}>\frac{cs}{4},
\end{aligned}\]
which contradicts \eqref{eq:26}.  This completes the proof.  
\end{proof}
\begin{remark}
In view of \eqref{eq:33-2},  $X^*$ is a continuous-time Markov chain with the minimal state space $\bN$ and the transition matrix $(p^*_{ij}(t))_{i,j\in \bN}$.  The essential range of $X^*$ is {equal to} $\bN\cup \{\infty\}$,  and $\infty$ is a fictitious state.   
\end{remark}

\section{Ray-Knight compactification of birth and death processes}\label{SEC3}

From now on, we assume that $\infty$ is an exit or regular.  Let $X=(X_t)_{t\geq 0}$ be a (Borel measurable) $Q$-process realized on the sample space $(\Omega, \mathscr F,  \bP)$ with transition function $(p_{ij}(t))_{i,j\in\bN}$.  
If the initial distribution is the Dirac measure $\delta_i$ for some $i\in \bN$,  we write $\bP$ as $\mathbf{P}_i$ and the expectation with respect to $\bP_i$ as $\mathbf{E}_i$.  
The transition semigroup of $X$ is denoted by $P_t$,  i.e., $P_t(i, B)=P_t1_B(i)=\mathbf{P}_i(X_t\in B)$ for $i\in \bN$ and $B\subset \bN$.  Let $(R_\alpha)_{\alpha>0}$ be the resolvent of $P_t$,  i.e.,  for any $f\in \mathcal{N}^+_b$,  the family of all bounded and positive Borel functions on $\bN$,   
\[
	R_\alpha f(i):=\mathbf{E}_i\int_0^\infty e^{-\alpha t} f(X_t)dt=\int_0^\infty e^{-\alpha t} P_tf(x)dt.  
\]
Particularly, for $i,j\in \bN$,  $P_t(i,\{j\})=p_{ij}(t)$, and we write
\[
	R_\alpha(i,\{j\}):=\int_0^\infty e^{-\alpha t} p_{ij}(t)dt.
\]
Let $\eta_n:=\inf\{t>0: X_t=n\}$ and
\begin{equation}\label{eq:30}
	\eta:=\lim_{n\rightarrow \infty}\eta_n.
\end{equation}
Then $\mathbf{P}_i(\eta<\infty)=1$ for $i\in \bN$, and the killed process of $X$ at $\eta$ is {equal to} the minimal $Q$-process; see,  e.g.,  \cite[\S5.1]{WY92}.  
For $\alpha>0$,  define a function
\begin{equation}\label{eq:31}
	u_\alpha(i):=\mathbf{E}_i e^{-\alpha \eta},\quad i\in \bN  
\end{equation}
with the convention $u_\alpha(\partial)=0$.  The following lemma is simple but important.  

\begin{lemma}\label{LM31}
Assume that $\infty$ is regular or an exit.  Let $u_\alpha$ be defined as \eqref{eq:31}.  Then 
\[
\lim_{i\rightarrow \infty} u_\alpha(i)=1. 
\]
\end{lemma}
\begin{proof}
This fact can be demonstrated by,  e.g.,  revisiting the proof of \cite[Proposition~7.9]{L22}.  Furthermore,  one may also refer to \cite[Theorem~7.1]{F59} or \cite[\S7.3, Theorem~1~(ii)]{WY92}. 
\end{proof}

We are now ready to {discuss} the Ray-Knight compactification (see,  e.g.,  \cite{R59,  K65, CW05,  GR75}) of $X$.  {This approach,  which has been previously discussed for general continuous-time Markov chains in \cite[Chapter 9]{CW05},  will be applied with a slightly different  family of $1$-\emph{supermedian functions} as follows:}
\[
	G:=\{R_1 1_{\{i\}}:  i\in \bN\cup\{\partial\}\}\cup \{u_1\}.  
\]
According to \cite{K65} (also see \cite[Lemma~8.24]{CW05}),  there exists a unique minimal convex cone $\mathcal S(G)$,  referred to as the \emph{Ray cone},  such that
\begin{itemize}
\item[(i)] $1_{\bN\cup \{\partial\}}\in \cS(G)$,  $G\subset \cS(G)$;
\item[(ii)] $\cS(G)$ is separable in $\cN^+_b$ and separates points of $\bN\cup\{\partial\}$; 
\item[(iii)] $f\in \cS(G)$ implies that $R_\alpha f\in \cS(G)$ for $\alpha>0$;
\item[(iv)] If $f,g\in \cS(G)$,  then $f\wedge g\in \cS(G)$. 
\end{itemize}
Here, the separable property means that there exists a countable subset of $\cS(G)$ whose closure in $\cN_b^+$ with respect to the uniform norm contains $\cS(G)$.  
By utilizing this Ray cone,  we can apply the Ray-Knight compactification technique to obtain the following result.  The terminologies of Ray resolvent and Ray process are {referred to},  e.g.,  \cite{CW05,  GR75}. 

\begin{theorem}\label{THM32}
There exists a compact metric space $\bF$ and a Ray resolvent $(\bar{R}_\alpha)_{\alpha>0}$ on $\bF$ such that 
\begin{itemize}
\item[(1)] $\bN\cup \{\partial\}$ is a dense Borel subset of $\bF$;  
\item[(2)] Every $h\in \cS(G)$ can be extended to a unique function $\bar{h}$ in $C(\bF)$;  
\item[(3)] $\overline{\cS(G)}-\overline{\cS(G)}$ is dense in $C(\bF)$,  where $\overline{\cS(G)}=\{\bar{h}: h\in \cS(G)\}$;
\item[(4)] For each $i\in \bN\cup\{\partial\}$ and $A\subset \bN\cup\{\partial\}$,  we have $\bar{R}_\alpha(i,  A)=R_\alpha(i, A)$.  
\end{itemize}
Furthermore,  there exists a c\`adl\`ag Ray process $\bar{X}=(\bar{X}_t)_{t\geq 0}$,  realized on the sample space $(\Omega, \sF, \mathbf{P})$,  on $\bF$ such that 
\begin{itemize}
\item[(5)] $\bar{X}$ admits $(\bar{R}_\alpha)_{\alpha>0}$ as the resolvent;
\item[(6)] For a.e.  $t$,  $\mathbf{P}(\bar{X}_t=X_t)=1$.  
\end{itemize}
\end{theorem}
\begin{remark}
The enlarged metric space $\bF$ and Ray resolvent $\bar{R}_\alpha$ depend only on the choice of $G$.  They {do not depend on} the initial distribution of $X$.  In the context of general continuous-time Markov chains, $\bar{X}$ is not necessarily a modification of $X$ and particularly,  $\bar{X}$ may have a different initial distribution from $X$.  Fortunately, this is not the case we encounter; see Lemma~\ref{LM33}~(1). 
\end{remark}

According to \cite{CW05},  $\bar{X}$ is referred to as the \emph{Ray-Knight compactification} of $X$.  The set of \emph{non-branching points} for this Ray process is defined as
\[
	D:=\{x\in \bF: \alpha \bar{R}_\alpha f(x)\rightarrow f(x),  \alpha\uparrow \infty,  \forall f\in C(\bF)\}.  
\]
The set $B:=\bF\setminus D$ is called the set of \emph{branching points}.  
Let $(\bar{P}_t)_{t\geq 0}$ be the transition semigroup of $\bar{X}$,  which is called a \emph{Ray semigroup}.  It is important to note that $\bar{P}_0(x,  \cdot)=\delta_x$ does not hold for $x\in B$, but it holds for $x\in D$. 

The following lemma is crucial to our subsequent considerations.  We want to emphasize that the identity in the fourth assertion refers not only to the equality of set elements but also to the equality of topologies.  

\begin{lemma}\label{LM33}
\begin{itemize}
\item[(1)] For any $t\geq 0$,  $\mathbf{P}(\bar X_t=X_t)=1$.  
\item[(2)] $\bN\cup\{\partial\}\subset D$.  
\item[(3)] $\mathbf{P}$-a.s.,  $\bar{X}_t=X_t$ for a.e.  $t$.  
\item[(4)] $\bF=\bN\cup \{\infty\}\cup \{\partial\}$.  
\end{itemize}
\end{lemma}
\begin{proof}
\begin{itemize}
\item[(1)] Fix $t\geq 0$.  We first note that 
\begin{equation}\label{eq:43-2}
	\lim_{h\downarrow 0}\bP(X_{t+h}=X_t)=\lim_{h\downarrow 0}\sum_{i, k\in \bN\cup\{\partial\}} \bP(X_0=i)p_{ik}(t)p_{kk}(h)=1.  
\end{equation}
Using \eqref{eq:43-2} and Theorem~\ref{THM32}~(6),  
one can choose a decreasing sequence $h_n\downarrow 0$ such that $\bP(\bar{X}_{t+h_n}=X_{t+h_n})=1$ and $\bP(X_{t+h_n}=X_t)>1-1/2^n$.     Let $A_n:=\{X_{t+h_n}\neq X_t\}$ and $N:=\cap_{k\geq 1}\cup_{n\geq k} A_n$.  It is clear that $\bP(N)=0$.   For $\bP$-a.s. $\omega\notin N$,  there exists $k\geq 1$ such that for any $n\geq k$, 
\[
	\bar X_{t+h_n}(\omega)=X_{t+h_n}(\omega)=X_t(\omega).
\]  
Since $\bar{X}$ is c\`adl\`ag,  it follows that $\bP(X_t=\bar{X}_t)=1$. 
\item[(2)]  Fix $i\in \bN\cup\{\partial\}$.  Theorem~\ref{THM32}~(4) yields that $\bar{P}_t(i,\{i\})=p_{ii}(t)$ for a.e.  $t$.  Since the function $t \mapsto \bar{P}_t f(i)$ is right continuous for $f\in C(\bF)$ (see, e.g.,  \cite[Theorem~8.2]{CW05}),  it follows that $\bar{P}_t(i,\cdot)$ converges weakly to $\bar{P}_0(i,\cdot)$ on $\bF$ as $t\downarrow 0$.   Take a sequence $t_n\downarrow 0$ such that $\bar{P}_{t_n}(i,\{i\})=p_{ii}(t_n)$.  Note that $\{i\}$ is closed in $\bF$.  We have
\[
	\bar{P}_0(i,\{i\})\geq \limsup_{n\rightarrow \infty}\bar{P}_{t_n}(i,\{i\})=\lim_{n\rightarrow \infty} p_{ii}(t_n)=1.  
\]
Therefore,  $i\in D$. 
\item[(3)] By applying Theorem~\ref{THM32}~(6) and the Fubini theorem, we obtain
\[
	\mathbf{E}\int_0^\infty 1_{\{\bar{X}_t\neq X_t\}}dt=\int_0^\infty \mathbf{P}(\bar{X}_t\neq X_t)dt=0.  
\]
Thus,  the desired assertion follows.  
\item[(4)] Consider a countable subset $\cS_0\subset \cS(G)$ with $u_1\in \cS_0$, such that $\bar{\cS}_0:=\{\bar{f}: f\in \cS_0\}$ separates points of $\bF$,  where $\bar{f}\in C(\bF)$ denotes the extension of $f$.  Note that $\mathbf{P}(\eta<\infty)=1$.  Since $\bar{X}$ is a c\`adl\`ag process on $\bF$,  it follows that for $\mathbf{P}$-a.s.  $\omega\in \Omega$,  the limit 
\[
	\lim_{t\uparrow \eta} \bar{f}(\bar{X}_t)=\bar{f}(\bar{X}_{\eta-}),\quad \forall f\in \cS_0
\]
exists.    The previous assertion implies that
\begin{equation}\label{eq:32}
\bar{f}(\bar{X}_{\eta-})=\lim_{t\rightarrow \eta} f(X_t).
\end{equation}
Note that $X$ must pass through $n$ before reaching $n+1$.  Hence, the right-hand side of \eqref{eq:32} is also equal to $\lim_{n\rightarrow \infty} f(n)$.  This implies that $\bar{f}(\bar{X}_{\eta-})=\lim_{n\rightarrow \infty} f(n)$ for $f\in \cS_0$.  Thus,  $\bar{f}(\bar{X}_{\eta-})$ is independent of $\omega \in \Omega$.  Since $\bar{\cS}_0$ separates points of $\bF$,  there exists an element $\xi\in \bF$ such that $\bar{X}_{\eta_-}=\xi$,  $\mathbf{P}$-a.s.  In particular, 
\[
	\bar{f}(\xi)=\lim_{n\rightarrow \infty} f(n)=\lim_{n\rightarrow \infty} \bar{f}(n),\quad f\in \cS_0.  
\]

Next, we note that $\xi\notin \bN\cup\{\partial\}$.  This is because $\bar{u}_1(i)=u_1(i)$ for any $i\in \bN\cup\{\partial\}$, while,  in view of Lemma~\ref{LM31}, $\bar{u}_1(\xi)=\lim_{n\rightarrow \infty} u_1(n)>\bar{u}_1(i)$.  

Finally, we show that $\bF=\bN\cup\{\partial\}\cup \{\xi\}$ and $n\rightarrow \xi$ in $\bF$ as $n\rightarrow \infty$.  (This implies that $\xi$ is {equal to} the Alexandroff compactification point of $\bN$.  Therefore,  $\xi=\infty$ and $\bF=\bN\cup \{\infty\}\cup\{\partial\}$.) Now, we argue by contradiction and suppose that $\xi'\in \bF\setminus \left(\bN\cup\{\partial\}\right)$ and $\xi'\neq \xi$.  According to Theorem~\ref{THM32}~(1),  there exists a sequence $\{n_k:k\geq 1\}\subset \bN$ such that $n_k\rightarrow \xi'$ in $\bF$ as $k\rightarrow \infty$.  Hence, we have
\[
	\bar{f}(\xi')=\lim_{k\rightarrow \infty} \bar{f}(n_k)=\bar{f}(\xi),\quad \bar{f}\in \cS_0.  
\]
Since $\cS_0$ separates points of $\bF$,  we must have $\xi'=\xi$, which leads to a contradiction.  Similarly, we can also conclude that $n\rightarrow \xi$ in $\bF$ as $n\rightarrow \infty$.  
\end{itemize}
This completes the proof. 
\end{proof}
\begin{remark}
If we {do not include} $u_1$ in $G$ and obtain the Ray-Knight compactification in a similar manner, it is possible that $\xi \in \mathbb{N} \cup \{\partial\}$ in the proof of the fourth assertion (specifically, when $X$ is a Doob process with an entering distribution represented by a point measure at $\xi$).  {In this particular case}, $\mathbf{F} = \mathbb{N} \cup \{\partial\}$ (although their topologies are not the same) and $n \rightarrow \xi$ as $n \rightarrow \infty$ in terms of the topology of $\mathbf{F}$.  More details regarding this possibility can be found in \cite[Remark~9.13]{CW05}.
\end{remark}

The first assertion in this lemma indicates that $\bar{X}$ is also a $Q$-process.  Referring to the terminology in Section~\ref{SEC22}, $\bar X$ is called a Ray $Q$-process. 
Furthermore,  we can easily obtain the following corollary,  {since a Ray process possesses the strong Markov property. }

\begin{corollary}
Every $Q$-process admits a {c\`adl\`ag} modification on $\bN\cup\{\infty\}\cup\{\partial\}$ {satisfying} the strong Markov property. 
\end{corollary}

The Ray-Knight compactification is helpful to understanding a $Q$-process and particularly to clarifying certain intricate concepts.  To illustrate this,  we will focus on an important notion of \emph{flying times}.  A time $t$ is referred to as a \emph{jumping time} (for $X_\cdot(\omega)$) if there exists a constant $\varepsilon>0$ and two distinct states $i,j\in \bN$ such that $X_s(\omega)=i$ for $s\in (t-\varepsilon,t)$ and $X_t(\omega)=j$.  The set of flying times $L(\omega)$ for $\omega\in \Omega$ is defined as
\[
	L(\omega):= \{t\in (0,\infty): \forall \varepsilon>0,  [t-\varepsilon, t]\text{ contains {infinitely many} jumping times for }X_\cdot(\omega)\}.  
\]
This definition is attributed to \cite[\S6.1]{WY92}.  
Recall that $\bar{X}$ is c\`adl\`ag on $\bF$,  which implies that for $\mathbf{P}$-a.s.  $\omega\in \Omega$,  $\bar{X}_{t-}$ exists in $\bF$.  By using this modification,   the set of \emph{flying times} of $X$ can be expressed as the set of times $t$ for which $\bar{X}_{t-}=\infty$. 

\begin{corollary}\label{COR35}
For $\mathbf{P}$-a.s.  $\omega\in \Omega$,
\[
	L(\omega)=\{t\in (0,\infty): \bar{X}_{t-}(\omega)=\infty\}.  
\]
\end{corollary}
\begin{proof}
If $\bar{X}_{t-}(\omega)=\infty$,  Lemma~\ref{LM33}~(3) implies that $\bN\ni X_{t_n}(\omega)\rightarrow \infty$ for a certain sequence $t_n\uparrow t$.  Clearly, $t$ is a flying time.  Conversely,  suppose $t\in L(\omega)$.  By way of contradiction, assume that $\bar{X}_{t-}(\omega)=j\in \bN\cup\{\partial\}$.  Then, it follows from Lemma~\ref{LM33}~(3) that there exists a constant $\varepsilon>0$ such that $X_{s}(\omega)=j$ for a.e.  $s\in [t-\varepsilon, t)$.  {Since $j$ is a stable state,  meaning that  $q_j<\infty$ where $q_j$ is located on the diagonal of $Q$, } we can apply \cite[II \S7, Theorem~4]{C67} to conclude that $X_s(\omega)=j$ for all $s\in [t-\varepsilon,t)$.  However, this contradicts the assumption that $t\in L(\omega)$.  Therefore,  the proof is complete.   
\end{proof}



\section{Doob processes}\label{SEC4}

The notion of \emph{Doob process} is {referred to},  e.g.,  \cite[\S2.3]{WY92}.  It is a well-known $Q$-process that returns to $\bN\cup \{\partial\}$ through non-local jumps at each flying time until reaching the trap $\partial$.  This process is uniquely determined by its returning distribution,  which is a specific probability measure $\pi$ on $\bN\cup \{\partial\}$.  For precision,  we also refer to it as the \emph{$(Q, \pi)$-Doob process}.  In particular,  when $\pi=\delta_{\partial}$,  the Dirac measure at $\partial$,  the $(Q,\pi)$-Doob process reduces to the minimal $Q$-process. 

{The main result of this section will demonstrate that a $Q$-process is a Doob process if and only if $\infty$ is {a branching point} for its Ray-Knight compactification. Furthermore, another goal is to provide a probabilistic construction of Doob processes using the celebrated {Ikeda-Nagasawa-Watanabe piecing out procedure} as reviewed in Appendix \ref{APPA}.
To explain the space $\mathbb{N}\cup \{\partial\}$ for the piecing out  method, we note that this construction begins with the minimal $Q$-process $X^{\text{min}}$. According to \cite{INW66}, the trap point $\partial$ for $X^{\text{min}}$ (as well as for the process obtained by piecing out) should be viewed as the Alexandroff compactification point of $\mathbb{N}$. While this viewpoint is inconsistent with our default notation, it should not cause misunderstanding since we consider processes with the same transition matrix to represent the same $Q$-process.


Let $X=(X_t)_{t\geq 0}$ be a $Q$-process whose Ray-Knight compactifiction on $\bF=\bN\cup \{\infty\}\cup\{\partial\}$ is denoted by $\bar{X}=(\bar{X}_t)_{t\geq 0}$.  The other notations concerning,  e.g.,  the transition semigroups,  the resolvents,  the branching or non-branching points,   remain the same as in Section~\ref{SEC3}.  Note that  if $\infty\in B$,  then $\bar P_0(\infty, \cdot)$ is a probability measure on $\bN\cup\{\partial\}$ (not on $\bF$!); see,  e.g.,  \cite[Proposition~8.8~(i)]{CW05}. }



\begin{theorem}\label{THM41}
Let $\pi$ be a probability measure on $\bN\cup\{\partial\}$.  The following are equivalent:
\begin{itemize}
\item[(1)] $X$ is the $(Q, \pi)$-Doob process;
\item[(2)] $X$ is the piecing out of $X^\text{min}$ with respect to $\pi$;
\item[(3)] $\infty$ is branching for $\bar X$ and $\bar P_0(\infty,\cdot)=\pi$.  
\end{itemize}
\end{theorem}
\begin{proof}
We first establish the equivalence between the statements (1) and (2).  
It is sufficient to compute the resolvent $(\tilde{R}_\alpha)$ of the piecing out process $\tilde{X}=(X^\text{min}, \pi)$.  Fix $\alpha>0$ and $i,j\in \bN$.  Adopt the notations from \eqref{SQ2XTW} to \eqref{SQ2XOF}.  Write $\pi_k:=\pi(\{k\})$ and $\Phi_{ij}(\alpha):=R^\text{min}_\alpha(i, \{j\})$ for convenience.  Since $\tau(\tilde{\mathtt w}):=\zeta(\omega_1)$ is an $\tilde{\cF}_t$-stopping time (see the remark before Theorem~1 in \cite{INW66}),   we can use the strong Markov property of $\tilde{X}$ and the Dynkin formula to obtain that
\begin{equation}\label{eq:41}
	\tilde{R}_{ij}(\alpha):=\tilde{R}_\alpha(i, \{j\})=\Phi_{ij}(\alpha)+\tilde{\mathbf{E}}_i \left[e^{-\alpha \tau}\tilde{R}_\alpha (\tilde{X}_\tau,  \{j\}) \right].
\end{equation}
According to \eqref{SQ2XTW} and  \eqref{eq:A3},  we have
\[
\begin{aligned}
	\tilde{\mathbf{E}}_i \left[e^{-\alpha \tau}\tilde{R}_\alpha (\tilde{X}_\tau,  \{j\}) \right]&=\int_{(\omega_1, y_1)\in \Omega\times (\bN\cup\{\partial\})} e^{-\alpha \zeta(\omega_1)}\tilde{R}_\alpha(y_1, \{j\})\mathbf{P}_{i}(d\omega_1)\nu(\omega_1,dy_1) \\
	&=u_\alpha(i)\sum_{k\in \bN\cup\{\partial\}} \tilde{R}_{kj} (\alpha)\pi_k. 
\end{aligned}
\]
Integrating \eqref{eq:41} with respect to $\pi$ over $i$,  we get
\[
\sum_{k\in \bN\cup\{\partial\}} \tilde{R}_{kj} (\alpha)\pi_k=\frac{\sum_{k\in \bN\cup\{\partial\}}\Phi_{kj} (\alpha)\pi_k }{1-\pi(u_\alpha)}.
\]
It follows that
\begin{equation}\label{eq:42}
\tilde{R}_{ij}(\alpha)=\Phi_{ij}(\alpha)+ u_\alpha(i)\cdot \frac{\sum_{k\in \bN\cup\{\partial\}}\Phi_{kj} (\alpha)\pi_k }{1-\pi(u_\alpha)}.  
\end{equation}
On account of Theorem~\ref{THMB1}~(3),  $\tilde{R}_\alpha$ is {equal to} the resolvent of the $(Q,\pi)$-Doob process.  Therefore, this equivalence can be eventually concluded. 

{Next, suppose that $X$ is the $(Q,\pi)$-Doob process. According to Theorem~\ref{THM32}~(4),  $\bar{R}_\alpha 1_{\{j\}}(i)$ is equal to  \eqref{eq:42} for all $i,j\in \bN$.   Letting $i\rightarrow \infty$ in \eqref{eq:42} and noting that $\Phi_{ij}(\alpha)\rightarrow 0$ and $u_\alpha(i)\rightarrow 1$ as $i\rightarrow \infty$ (see Lemma~\ref{LM31}),  one obtains
\[
	\bar R_\alpha1_{\{j\}}(\infty)=\frac{\sum_{k\in \bN\cup\{\partial\}}\Phi_{kj} (\alpha)\pi_k }{1-\pi(u_\alpha)}.
\]
A straightforward computation then yields that $\alpha \bar R_\alpha 1_{\{j\}}(\infty)\rightarrow \pi_j\neq 1_{\{j\}}(\infty)$ as $\alpha\rightarrow \infty$.  Thus, $\infty$ is branching for $\bar X$.  Mimicking \eqref{eq:41},  we can obtain
\[	\bar R_\alpha 1_{\{j\}}(i)=\Phi_{ij}(\alpha)+{\mathbf{E}}_i \left[e^{-\alpha\eta}\bar R_\alpha 1_{\{j\}} (\bar X_\eta) \right],
\]
where $\eta$ is defined as \eqref{eq:30}.  It follows from \cite[Proposition~8.7]{CW05} that
\[
{\mathbf{E}}_i \left[e^{-\alpha\eta}\bar R_\alpha 1_{\{j\}} (\bar X_\eta) \right]=\mathbf{E}_i\left[e^{-\alpha \eta} \mathbf{E}_i \left[\bar R_\alpha 1_{\{j\}} (\bar X_\eta)|\sF_{\eta-} \right] \right]=u_\alpha(i)\pi'(\bar R_\alpha 1_{\{j\}}),
\]
where $\pi':=\bar P_0(\infty, \cdot)$.  Hence, we can further obtain
\begin{equation}\label{eq:43}
	\bar R_\alpha 1_{\{j\}}(i)=\Phi_{ij}(\alpha)+ u_\alpha(i)\cdot \frac{\sum_{k\in \bN\cup\{\partial\}}\Phi_{kj} (\alpha)\pi'_k }{1-\pi'(u_\alpha)},
\end{equation}
where $\pi'_k:=\pi'(\{k\})$.  Comparing \eqref{eq:43} with \eqref{eq:42}, one has
\begin{equation}\label{eq:44}
	\frac{\sum_{k\in \bN\cup\{\partial\}}\alpha \Phi_{kj} (\alpha)\pi_k }{1-\pi(u_\alpha)}= \frac{\sum_{k\in \bN\cup\{\partial\}}\alpha\Phi_{kj} (\alpha)\pi'_k }{1-\pi'(u_\alpha)}.  
\end{equation}
Letting $\alpha\rightarrow \infty$ in \eqref{eq:44},  we have $\pi_k=\pi'_k$.  Therefore, the conclusion $\bar P_0(\infty, \cdot)=\pi$ follows. 
}

Finally, suppose that $\infty$ is a branching point for $\bar{X}$ and $\bar{P}_0(\infty, \cdot)=\pi$.  Repeating the arguments from the previous steps, one can readily obtain \eqref{eq:43} with $\pi'=\pi$. Thus, $X$ is indeed the $(Q,\pi)$-Doob process. This completes the proof.
\end{proof}

{
It is shown in \cite[Chapter I, Theorem~9.13]{S88} that the restriction of a Ray semigroup to the set of non-branching points admits a realization as a Borel right process. For the definition regarding the realization of a semigroup, the readers are referred to \cite{S88}.  Additionally, in the absence of branching points,  Ray processes are Feller processes; see e.g. \cite[III (37.1)]{RW00}. Given these results established in the literature, the following corollary readily follows.

\begin{corollary}\label{COR52}
The transition matrix of a Doob process has a realization on $\bN\cup\{\partial\}$ as a Borel right process.  Every Ray $Q$-process that is not a Doob process is a Feller process on $\bN\cup\{\infty\}\cup\{\partial\}$. 
\end{corollary}}

\section{Infinitesimal generators of Feller birth and death processes}

{The subsequent sections of this paper will focus on $Q$-processes with Ray-Knight compactifications that do not display branching behaviour at $\infty$.  Following Lemma~\ref{LM33}, we will not distinguish between the $Q$-process and its Ray-Knight compactification,  assuming that the $Q$-process $X$ is a Ray process on $\mathbf{F}=\mathbb{N}\cup \{\infty\}\cup\{\partial\}$ without further specification.  When $\infty$ is non-branching for $X$,  this process is called a Feller $Q$-process, as Corollary~\ref{COR52} demonstrates that it represents a Feller process on $\bF$.
}

{\subsection{Maximal generalized second order differential operator}
Before proceeding, we will introduce a \emph{maximal (discrete) generalized second order differential operator}; see,  e.g., \cite{F59}.  The birth and death matrix $Q$, as defined in \eqref{eq:11},  will be used in the following discussion.  By an abuse of notation, define a function $QF$ on $\bN$ for every function $F$ on $\bN$ as
\begin{equation}\label{eq:72}
	QF(k):=a_k F(k-1)-q_kF(k)+b_k F(k+1),\quad k\geq 0,
\end{equation}
where $F(-1):=0$.  {This definition provides the desired operator, which is also denoted by $Q$.}
In most cases,  we focus on the restriction of this operator to the set $C(\bN\cup \{\infty\})$ of all continuous functions on $\bN\cup \{\infty\}$.  
A function $F$ defined on $\bN$ can be extended to a function in $C(\bN\cup\{\infty\})$ if and only if $F(\infty):=\lim_{k\uparrow\infty}F(k)$ exists.  In this case, we represent its extension as the same symbol $F$.   Define a subset of $C(\bN\cup \{\infty\})$ as follows:
\[
\cD(Q):=\{F\in C(\bN\cup\{\infty\}): QF\in C(\bN\cup\{\infty\})\},
\]
which will be used later. 

It is meaningful to provide an alternative expression for this operator $Q$.  
Before that, we introduce some new notations for a function $F$ on $\bN$.  Let us define
\begin{equation}\label{eq:61-2}
	F^+(k):=\frac{F(k+1)-F(k)}{c_{k+1}-c_k},\quad k\geq 0
\end{equation}
and $F^+(\infty):=\lim_{k\uparrow \infty}F^+(k)$ if it exists.  In order to maintain consistency,  we adopt the convention
\[
	F^+(-1):=0,
\]
which can be regarded as a \emph{fictitious} Neumann boundary condition at $0$.  With this in mind,  we further define
\[
	D_\mu F^+(k):=\frac{F^+(k)-F^+(k-1)}{\mu_k},\quad k\geq 0.  
\]
The following result, which shows why the operator $Q$ is called a ``second order differential operator",  is well-known.}

\begin{lemma}\label{LM71}
For every function $F$ on $\bN$,  
\begin{equation}\label{eq:73}
	QF=\frac{1}{2}D_\mu F^+.  
\end{equation}
Furthermore,  $\cD(Q)$ is dense in $C(\bN\cup\{\infty\})$.
\end{lemma}
\begin{proof}
The equality \eqref{eq:73} can be easily verified through a straightforward computation. To demonstrate the denseness of $\cD(Q)$, it is sufficient to observe that $1_{\bN\cup\{\infty\}}\in \cD(Q)$ and that $C_c(\bN):=\{F\in C(\bN\cup \{\infty\}): \exists n\in \bN \text{ s.t. } F(k)=0,\forall k\geq n\}$ is a subset of $\cD(Q)$. This concludes the proof.
\end{proof}

\subsection{{Infinitesimal generators of Feller $Q$-processes}}

{Let $X$ be a Feller $Q$-process on $\bN\cup\{\infty\}\cup \{\partial\}$.  Its semigroup $(P_t)_{t\geq 0}$ acts on $C(\bN\cup \{\infty\})$ as a strongly continuous contraction semigroup. This means that $P_tC(\bN\cup \{\infty\})\subset C(\bN\cup \{\infty\})$ and $P_tF\rightarrow F$ in $C(\bN\cup \{\infty\})$ as $t\downarrow 0$ for any $F\in C(\bN\cup \{\infty\})$.  Then, it is well known that a function $F\in C(\bN\cup \{\infty\})$ is said to belong to the domain $\cD(\sL)$ of the infinitesimal generator of $X$ if the limit
\[
		\sL F:=\lim_{t\downarrow 0}\frac{P_t F-F}{t}
\]
exists in $C(\bN\cup \{\infty\})$.  The operator $\sL: \cD(\sL)\rightarrow C(\bN\cup \{\infty\})$ thus defined is called the infinitesimal generator of the process $X$.  

The aim of this section is to characterize the infinitesimal generator of the Feller $Q$-process.  To accomplish this aim, we will utilize the triplet $(\nu, \gamma, \beta)$ proposed by Wang and Yang in \cite{WY92} to characterize the resolvent matrix of a $Q$-process. 
In this triplet, $\nu$ is a positive measure on $\bN$, and $\gamma$ and $\beta$ are two non-negative constants. The relevant details about this characterization can be found in Appendix~\ref{APPB}.  With the help of Corollary~\ref{COR52} and Theorem~\ref{THMB1}, the following lemma is straightforward. 

\begin{lemma}\label{LM62}
Let $X$ be a $Q$-process with the triple $(\nu, \gamma, \beta)$ determining its resolvent matrix.  Then,  $X$ is a Feller $Q$-process on $\bN\cup \{\infty\}\cup \{\partial\}$,  if and only if \eqref{eq:B1},  \eqref{eq:B4} and \begin{equation}\label{eq:71}
|\nu|=\infty\text{ or }\beta>0
\end{equation}
are satisfied.
\end{lemma}
}




We are now {ready} to present the main result of this section.  In the following expression \eqref{eq:74},  when $\infty$ is an exit,  we set $\beta F^+(\infty):=0$, even if $F^+(\infty)$ does not exist.  Otherwise, $F^+(\infty)$ and the sum $\sum_{k\geq 0}(F(\infty)-F(k))\nu_k$ are assumed to be well defined for any $F$ in this family.  

\begin{theorem}\label{THM62}
A linear operator $\sL$ with domain $\cD(\sL)$ on $C(\bN\cup\{\infty\})$ is the infinitesimal {generator} of a Feller $Q$-process $X$ on $\bN\cup\{\infty\}\cup \{\partial\}$,  if and only if there exists a unique (up to a multiplicative constant) triple $(\nu, \gamma,\beta)$ as described in Lemma~\ref{LM62} such that
\begin{equation}\label{eq:74}
\begin{aligned}
\cD(\sL)=\bigg\{F\in \cD(Q): { \frac{\beta}{2}} F^+(\infty)+\sum_{k\geq 0}(F(\infty)-F(k))\nu_k +\gamma F(\infty)=0\bigg\} 
\end{aligned}
\end{equation}
and $\sL F=QF$ for $F\in \cD(\sL)$.  {In this case,} the resolvent matrix of $X$  is given by \eqref{eq:B2} using this triple.  
\end{theorem}
\begin{proof}
Let $(\nu, \gamma,\beta)$ be such a triple in Lemma~\ref{LM62}, and let $X$ be the corresponding Feller $Q$-process on $\bN\cup\{\infty\}\cup \{\partial\}$ with semigroup denoted by $(P_t)_{t\geq 0}$ and resolvent denoted by $(R_\alpha)_{\alpha>0}$. Define $\sL:=Q$ with the domain $\cD(\sL)$ as given in \eqref{eq:74}.  It suffices to prove that for any $f\in C(\bN\cup\{\infty\})$, we have $R_\alpha f\in \cD(\sL)$ and that $F:=R_\alpha f$ is the unique solution of the equation
\begin{equation}\label{eq:78}
	\alpha G-\sL G=f.  
\end{equation}

We first note that
\begin{equation}\label{eq:75}
\begin{aligned}
	F(k)&=R^\text{min}_\alpha f(k)+u_\alpha(k)\cdot \frac{\sum_{i\geq 0}\nu_i R^\text{min}_\alpha f(i)+\beta\sum_{i\geq 0} u_\alpha(i)f(i)\mu_i}{\gamma+\sum_{i\geq 0}\nu_i(1-u_\alpha(i))+\beta\alpha\sum_{i\geq 0}\mu_i u_\alpha(i)} \\
	&=:R^\text{min}_\alpha f(k)+u_\alpha(k)\cdot M(f),
\end{aligned}\end{equation}
where $M(f)$ is a constant depending on $(\nu, \gamma, \beta)$ and $f$.  In addition,  in view of \cite[Lemma~9.1]{F59} and Lemma~\ref{LM31},  $F\in C(\bN\cup\{\infty\})$ with $F(\infty)=M(f)$.  

Secondly,  we show that $F\in \cD(Q)$.  Since $X$ is {a Feller process},  there exists a function $h\in C(\bN\cup\{\infty\})$ such that
\[
	\lim_{t\downarrow 0}\sup_k\left|\frac{P_tF(k)-F(k)}{t}-h(k)\right|=0.  
\]
It follows from \cite[II \S3, Corollary]{C67} that
\[
	\left|\frac{P_tF(k)-F(k)}{t}-QF(k)\right|\leq \left(\sum_{j}\left| \frac{p_{kj}(t)-\delta_{kj}}{t}-q_{kj} \right| \right) \sup_{j}|F(j)|\rightarrow 0
\]
as $t\downarrow 0$ for any $k\geq 0$.  Thus,  we have $QF=h\in C(\bN\cup\{\infty\})$.  

Thirdly,  we derive the boundary condition for $F$ and establish that $F\in \cD(\sL)$.  We will only consider the case where $\infty$ is regular, as the other case where $\infty$ is an exit can be treated similarly.  According to \cite[Theorem~8.1]{F59},  it is known that $(R^\text{min}_\alpha f)^+(\infty):=\lim_{k\uparrow \infty} (R^\text{min}_\alpha f)^+(k)$ exists and can be expressed as 
\[
	(R^\text{min}_\alpha f)^+(\infty)=-{2}\sum_{j\in \bN} u_\alpha(j)f(j)\mu_j.
\]
It should be noted that $u_\alpha^+(\infty)=\lim_{k\uparrow \infty}u^+_\alpha(k)<\infty$ due to \cite[(7.4)]{F59}.  Therefore, we can use \eqref{eq:75} to obtain that 
\begin{equation}\label{eq:76}
	F^+(\infty)=-{2}\sum_{j\in \bN} u_\alpha(j)f(j)\mu_j+M(f)\cdot u^+_\alpha(\infty).  
\end{equation}
On the other hand,  we have
\[
	F(\infty)-F(k)=M(f)(1-u_\alpha(k))-R^\text{min}_\alpha f(k),
\]
which implies that
\begin{equation}\label{eq:77}
\sum_{k\geq 0}(F(\infty)-F(k))\nu_k=M(f)\sum_{i\geq 0}\nu_i(1-u_\alpha(i))-\sum_{i\geq 0}\nu_i R^\text{min}_\alpha f(i). 
\end{equation}
By using \eqref{eq:76},  \eqref{eq:77},  $F(\infty)=M(f)$ and noting that $u^+_\alpha(\infty)={2}\alpha \sum_{j} u_\alpha(j)\mu_j$ (see,  e.g.,  \cite[(7.5)]{F59}),  we can conclude that 
\begin{equation}\label{eq:79}
{\frac{\beta}{2}}\cdot F^+(\infty)+\sum_{k\geq 0}(F(\infty)-F(k))\nu_k +\gamma\cdot F(\infty)=0.  
\end{equation}
Based on the results obtained in the previous two steps,  we can establish that $F\in \cD(\sL)$.  

Finally,  we prove that $F$ is the unique solution of \eqref{eq:78}.  
By Lemma~\ref{LM71} and \cite[Theorem~9.1]{F59},  it follows that 
\[
	\alpha R^\text{min}_\alpha f - Q(R^\text{min}_\alpha f)=f.  
\]
According to \cite[Theorem~7.1]{F59},  we have
\[
\alpha u_\alpha -Q u_\alpha=0.  
\]	
Since $F=R^\text{min}_\alpha f+M(f)\cdot u_\alpha\in \cD(\sL)$,  we can conclude that
\[
	\alpha F-\sL F=\alpha F-QF=f.  
\]
It remains to verify the uniqueness of the solutions of \eqref{eq:78}.  Argue by contradiction and suppose that $F_0\in \cD(\sL)$ with $Q F_0=\alpha F_0$ and $F_0\neq 0$.  By using \cite[Theorem~7.1]{F59},  it can be deduced that $F_0$ must be equal to $K\cdot u_\alpha$ for some constant $K\neq 0$.   However,  since $u_\alpha$ is increasing, \eqref{eq:79} does not hold for $u_\alpha$ unless $\gamma=\beta=\nu_k=0$.  This implies that $u_\alpha\notin \cD(\sL)$,  leading to a contradiction.  Therefore,  the proof is complete. 
\end{proof}

When $\infty$ is an exit,  we have $\beta=0$ and therefore $|\nu|=\infty$ in Theorem~\ref{THM62}.  {In this case,} the boundary condition in \eqref{eq:74} is written as follows:
\begin{equation}\label{eq:610}
\sum_{k\geq 0}(F(\infty)-F(k))\nu_k +\gamma F(\infty)=0.
\end{equation}
On the other hand,   if $X$ is a Doob process,  meaning that $\beta=0, |\nu|<\infty$ and $|\nu|+\gamma>0$,  we can still follow the same procedure as the third step in the proof of Theorem~\ref{THM73} to derive \eqref{eq:610} for every $F=R_\alpha f$.  Particularly,  with regards to the minimal $Q$-process, it satisfies the Dirichlet boundary condition $F(\infty)=0$.  

\section{Symmetrizable birth and death processes}\label{SEC5}

Before addressing our probabilistic construction for general Feller $Q$-processes,  we will particularly concentrate on all symmetrizable ones in this section.  It is noteworthy that all symmetrizable $Q$-processes, except for the minimal process, are not Doob processes. Moreover, Theorem~\ref{THM25} {implies} the assumption that $\infty$ is regular.

Recall that $X^*$ represents the $(Q,1)$-process that corresponds to the basic Dirichlet form $(\sA^*,\sG^*)$,  which is regular on $L^2(\bN\cup\{\infty\}, \mu)$; see Section~\ref{SEC23}.  Since $\infty$ is regular,  $f(\infty)$ exists for every function $f\in \sG^*$.  This implies that $\sG^*\subset C(\bN\cup \{\infty\})$.   Let $\kappa\in [0,\infty)$ be a constant.  The following lemma is straightforward; see,  e.g.,  \cite[\S6.1]{FOT11}. 

\begin{lemma}\label{LM51}
Assume that $\infty$ is regular.  Let $\kappa\in [0,\infty)$.  Then, the quadratic form
\[
\begin{aligned}
&\sF^\kappa:=\sG^*,  \\
&\sE^\kappa(f,g):=\sA^*(f,g)+\kappa \cdot f(\infty)g(\infty),\quad f,g\in \sF^\kappa
\end{aligned}
\]
is a regular Dirichlet form on $L^2(\bN\cup\{\infty\},\mu)$.  
\end{lemma}

We denote the associated Markov process of $(\sE^\kappa, \sF^\kappa)$ by $X^\kappa$.  Obviously,  $X^\kappa=X^*$ for $\kappa=0$.  When $\kappa>0$,  this process can be obtained by killing $X^*$ at a certain flying time. 
{The constant $1/\kappa$ quantifies the mean of the random variable with an exponential distribution in the expression of this killing time, as seen in \eqref{eq:13}. } The minimal $Q$-process may be regarded as the uncovered case $\kappa=\infty$. 

Mimicking the proof of Proposition~\ref{PRO23},  we can easily demonstrate that 
\begin{equation}\label{eq:52}
	\mathbf{P}_i(X^\kappa_t=\infty)=0,\quad t\geq 0,  i\in \bN,
\end{equation}
which implies that
\begin{equation}\label{eq:51}
	p^\kappa_{ij}(t):=\mathbf{P}_i(X^\kappa_t=j),\quad t\geq 0,  i,j\in \bN
\end{equation}
defines a symmetrizable $Q$-process.  

\begin{lemma}
Let $(p^\kappa_{ij}(t))_{i,j\in \bN}$ be the $Q$-process defined as \eqref{eq:51} for $\kappa\in [0,\infty)$.  Then $X^\kappa$ is the Ray-Knight compactification of $(p^\kappa_{ij}(t))_{i,j\in \bN}$.  Particularly,  $\infty$ is {a non-branching point} for $X^\kappa$. 
\end{lemma}
\begin{proof}
With an abuse of notation, we write $C(\bN\cup\{\infty\})=\{f\in C(\bF): f(\partial)=0\}$.  This notation allows us to express any function $f\in C(\bF)$ as a decomposition:
\[
	f=(f-f(\partial)\cdot 1_{\bF})+f(\partial) \cdot 1_\bF=: f_0+ f(\partial) \cdot 1_\bF,
\]
where $f_0\in C(\bN\cup\{\infty\})$ and $f(\partial) \cdot 1_\bF$ is a constant function on $\bF$.  Let $(R^\kappa_\alpha)_{\alpha>0}$ denote the resolvent of $X^\kappa$.   Since $C(\bN\cup\{\infty\})\subset L^2(\bN\cup\{\infty\}, \mu)$,  it follows that $R^\kappa_\alpha C(\bN\cup\{\infty\})\subset \sF^\kappa \subset C(\bN\cup\{\infty\})$.  

Denote by $\bar{X}$ the Ray-Knight compactification of $(p^\kappa_{ij}(t))_{i,j\in \bN}$, and by $\bar{R}_\alpha$ its resolvent on $C(\bF)$.  To prove the statement,  it suffices to show that for any $\bar f\in C(\bN\cup\{\infty\})$, 
\[
	\bar{R}_\alpha \bar f(i)=R^\kappa_\alpha \bar{f}(i),  \quad i\in \bN\cup\{\infty\}.   
\]
Let $f:=\bar{f}|_\bN$.  For each $i\in \bN$,  it follows from Theorem~\ref{THM32}~(4),  Lemma~\ref{LM33} and \eqref{eq:52} that
\begin{equation}\label{eq:53}
\bar{R}_\alpha \bar f(i)=\int_0^\infty e^{-\alpha t}dt\sum_{j\in \bN} p^\kappa_{ij}(t)f(j)=\mathbf{E}_i \int_0^\infty e^{-\alpha t} \bar f(X^\kappa_t)dt=R^\kappa_\alpha \bar{f}(i).
\end{equation}
Note that $\bar{R}_\alpha \bar{f}\in C(\bF)$ and $R^\kappa_\alpha \bar{f}\in C(\bN\cup\{\infty\})$.  Letting $i\rightarrow \infty$ in \eqref{eq:53},  we further obtain $\bar{R}_\alpha \bar f(\infty)=R^\kappa_\alpha \bar{f}(\infty)$.  This completes the proof. 
\end{proof}

 In fact,  all symmetrizable $Q$-processes are given by \eqref{eq:51}.  That is the following.

\begin{theorem}\label{THM73}
Assume that $\infty$ is regular. 
Let $(p_{ij}(t))_{i,j\in\bN}$ be a $Q$-process {which is not minimal}.  Then,  it is symmetrizable if and only if there exists a constant $\kappa\in [0,\infty)$ such that 
\[
	p_{ij}(t)=p_{ij}^\kappa(t),\quad t\geq 0, i,j\in \bN,
\]
where $(p_{ij}^\kappa(t))_{i,j\in \bN}$ is defined as \eqref{eq:51}.  Particularly,  the branching set for $(p_{ij}(t))_{i,j\in \bN}$ is empty. 
\end{theorem}
\begin{proof}
Only the necessity needs to be proved.  Let $\bar{X}$ be the Ray-Knight compactification of a symmetrizable $Q$-process $(p_{ij}(t))_{i,j\in\bN}$.  We will first establish that the branching set for $\bar{X}$ is empty,  meaning that $\infty\notin B$.  Assume, by contradiction, that $D=\bN\cup\{\partial\}$.  According to \cite[Theorem~9.13]{S88},  the restriction $\bar{X}'$ of $\bar{X}$ to $D$ is a Borel right process,  where $D$ is endowed with the relative topology induced by $\bF$.  Define $\eta'_n:=\inf\{t>0: \bar{X}'_t=n\}$ and $\eta':=\lim_{n\rightarrow \infty}\eta'_n$.  Clearly, the killed process of $\bar{X}'$ at $\eta'$ is the minimal $Q$-process.  Since $(p_{ij}(t))_{i,j\in\bN}$ is symmetric with respect to $\mu$,  it follows that $\bar{X}'$ is also symmetric with respect to $\mu$.  
Thus, by \cite[Corollary~3.1.14]{CF12}, $\bar{X}'$ is a special, Borel standard process.  Particularly, $\bar{X}'$ is quasi-left-continuous.  Let $\zeta'$ be the lifetime of $\bar{X}'$.  The quasi-left-continuity of $\bar{X}'$ implies that on $\{\eta'<\zeta'\}$,  we have
\[
	\bar{X}'_{\eta'_n}\rightarrow \bar{X}'_{\eta'}\in \bN,
\]
while $\bar{X}'_{\eta'_n}=n$ and $\lim_{n\rightarrow \infty} n\notin\bN$.  Therefore, we must have $\eta'=\zeta'$,  so that $\bar{X}'$ (and hence $\bar{X}$) is identical to the minimal $Q$-process.  This contradicts the assumption that  $(p_{ij}(t))_{i,j\in\bN}$ is not minimal. 

Based on the above argument,  $\bar{X}$ is a Feller process on $\bN\cup\{\infty\}\cup \{\partial\}$.  According to Lemma~\ref{LM33}~(1),  $\bar{X}$ is symmetric with respect to $\mu$.  Denote by $(\bar{\sE},\bar{\sF})$ its quasi-regular Dirichet form on $L^2(\bN\cup\{\infty\}, \mu)$.  Since the part Dirichlet form of $(\bar{\sE},\bar{\sF})$ on $\bN$ is {equal to} the minimal Dirichlet form $(\sA^\text{min}, \sG^\text{min})$,  it follows that $(\bar{\sE},\bar{\sF})$ admits no killing inside $\bN$.  In other words,  the killing measure $\bar{k}$ of $(\bar{\sE},\bar{\sF})$ must be supported on $\{\infty\}$.  Thus, there exists a constant $\kappa \in [0,\infty)$ such that $\bar{k}=\kappa \cdot \delta_\infty$.  Handling the resurrection on $(\bar{\sE},\bar{\sF})$, we can obtain the resurrected Dirichlet form $(\bar{\sE}^\text{res}, \bar{\sF}^\text{res})$ for  $(\bar{\sE},\bar{\sF})$,  as described in \cite[Theorem~5.2.17]{CF12}.  Note that the associated Markov process $\bar{X}^\text{res}$ of $(\bar{\sE}^\text{res}, \bar{\sF}^\text{res})$ remains symmetric with respect to $\mu$ and admits no killing inside $\bN\cup\{\infty\}$.  Additionally,  the part process of $\bar{X}^\text{res}$ on $\bN$ is the minimal $Q$-process $X^\text{min}$. In other words,  $\bar{X}^\text{res}$ is a one-point extension of $X^\text{min}$ in the sense of \cite[Definition~7.5.1]{CF12}.  Clearly, the $(Q,1)$-process $X^*$ is also a one-point extension of $X^\text{min}$.  Applying \cite[Theorem~7.5.4]{CF12},  one can obtain that $(\bar{\sE}^\text{res}, \bar{\sF}^\text{res})=(\sA^*, \sG^*)$.  In view of \cite[Theorem~5.2.17]{CF12} again,  we eventually conclude that $(\bar{\sE},\bar{\sF})=(\sE^\kappa,\sF^\kappa)$, where the latter Dirichlet form is defined in Lemma~\ref{LM51}.  This completes the proof. 
\end{proof}

\section{Birth and death processes with only big non-local jumps}

Let $(p_{ij}(t))_{i,j\in \bN}$ be a Feller $Q$-process with the triple $(\nu, \gamma, \beta)$,  as characterized in Theorem~\ref{THM62}. 
In this section,  we will present a probabilistic construction of $(p_{ij}(t))_{i,j\in \bN}$ under the condition $|\nu|:=\sum_{k\geq 0}\nu_k<\infty$.  As stated in Theorem~\ref{THMB1}~(3),  if $\beta=0$,  then $(p_{ij}(t))_{i,j\in \bN}$ is a Doob process,  as discussed in Section~\ref{SEC4}.  The primary focus of this section is the case $\beta>0$.  According to \eqref{eq:B4},  the additional condition that $\infty$ is regular is still required.




\begin{theorem}\label{THM61}
{Assume that $\infty$ is regular.  Consider a Feller $Q$-process $(p_{ij}(t))_{i,j\in \bN}$ with the triple $(\nu, \gamma,\beta)$, where $|\nu|<\infty$ and $\beta>0$.  Then, $(p_{ij}(t))_{i,j\in \bN}$ is the piecing out of $X^\kappa$ with respect to $\pi$,  where $X^\kappa$ is the symmetrizable $Q$-process described in Section~\ref{SEC5} and 
\begin{equation}\label{eq:61}
	\kappa=\frac{|\nu|+\gamma}{\beta},\quad \pi(\{k\})=\frac{\nu_k}{|\nu|+\gamma}, \; k\in \bN,\quad \pi(\{\partial \})=\frac{\gamma}{|\nu|+\gamma}.
\end{equation}}
\end{theorem}
\begin{proof}
Denote by $\tilde{X}$ the piecing out of $X^\kappa$ with respect to $\pi$, where $\kappa$ and $\pi$ are given in \eqref{eq:61}.  Let $R^\kappa_\alpha$ and $\tilde{R}_\alpha$ be the resolvents of $X^\kappa$ and $\tilde{X}$, respectively.  Define
\[
	R^\kappa_{ij}(\alpha):=R^\kappa_\alpha(i, \{j\}),  \quad \tilde{R}_{ij}(\alpha):=\tilde{R}_\alpha(i, \{j\}),\quad \alpha>0,  i,j\in \bN. 
\]
The $(Q,\nu, \gamma,\beta)$-resolvent matrix is expressed as $\Psi_{ij}(\alpha)$ in \eqref{eq:B2}.  To obtain $\tilde{R}_{ij}(\alpha)=\Psi_{ij}(\alpha)$, we will complete the computation in several steps. 

Firstly, we prove that
\begin{equation}\label{eq:62}
	R^\kappa_\alpha(\infty, \{j\})=\frac{\mu_j u_\alpha(j)}{\kappa+\alpha \mu(u_\alpha)}.  
\end{equation}
To do this,  let $f:=1_{\{j\}}$.  Since $f, 1_{\bN\cup\{\infty\}}\in \sF^\kappa$,  it follows that
\begin{equation}\label{eq:63}
	\mu(f)=\sE^\kappa_\alpha(R^\kappa_\alpha f, 1_{\bN\cup\{\infty\}})=\kappa R^\kappa_\alpha f(\infty)+\alpha \mu(R^\kappa_\alpha f).  
\end{equation}
By the Dynkin's formula, we have
\begin{equation}\label{eq:64}
R^\kappa_\alpha f=R^\text{min}_\alpha f + R^\kappa_\alpha f(\infty) u_\alpha.
\end{equation}
Combining \eqref{eq:63} and \eqref{eq:64},  we obtain
\[
R^\kappa_\alpha f(\infty)=\frac{\mu(f-\alpha R^\text{min}_\alpha f)}{\kappa+\alpha \mu(u_\alpha)}=  \frac{(1_{\bN\cup\{\infty\}}-\alpha R^\text{min}_\alpha 1_{\bN\cup\{\infty\}}, f)_\mu}{\kappa+\alpha \mu(u_\alpha)}.
\]
It remains to note that $1_{\bN\cup\{\infty\}}-\alpha R^\text{min}_\alpha 1_{\bN\cup\{\infty\}}=u_\alpha$, which yields \eqref{eq:62}. 

Next,  following a similar argument to derive \eqref{eq:42}, we can obtain another equality:
\begin{equation}\label{eq:65}
	\tilde{R}_\alpha f(i)=R^\kappa_\alpha f(i)+\frac{\pi(R^\kappa_\alpha f)}{1-\pi(v_\alpha)}\cdot v_\alpha(i).
\end{equation}
In this equality,  $v_\alpha(i):=\mathbf{E}_i^\kappa e^{-\alpha \zeta^\kappa}$,  where $\zeta^\kappa$ represents the lifetime of $X^\kappa$.  Note that the strong Markov property of $X^\kappa$ implies that $v_\alpha(i)=c_\alpha u_\alpha(i)$,  where $c_\alpha:=\mathbf{E}^\kappa_\infty e^{-\alpha \zeta^\kappa}$.  

Finally,  by substituting \eqref{eq:62} and \eqref{eq:64} into \eqref{eq:65} and performing a straightforward computation,  we obtain 
\begin{equation}\label{eq:66}
\tilde{R}_\alpha f(i)=R^\text{min}_\alpha f(i)+u_\alpha(x)\cdot \frac{\mu_j u_\alpha(j)+c_\alpha \pi(R^\text{min}_\alpha f) \cdot \left(\kappa+\alpha \mu(u_\alpha)\right)}{(1-c_\alpha \pi(u_\alpha))(\kappa+\alpha \mu(u_\alpha))}.
\end{equation}
It follows from \eqref{eq:62} that
\[
1-c_\alpha=1-\mathbf{E}^\kappa_\infty e^{-\alpha \zeta^\kappa}=\alpha R^\kappa_\alpha 1_{\bN\cup\{\infty\}}(\infty)=\frac{\alpha \mu(u_\alpha)}{\kappa+\alpha \mu(u_\alpha)}.
\]
Hence, $c_\alpha=\kappa/(\kappa+\alpha \mu(u_\alpha))$.  Substituting this into \eqref{eq:66}, we arrive at
\begin{equation}\label{eq:67}
	\tilde{R}_\alpha f(i)=R^\text{min}_\alpha f(i)+u_\alpha(i)\cdot \frac{\kappa \pi(R^\text{min}_\alpha f)+\mu_j u_\alpha(j)}{\kappa-\kappa \pi(u_\alpha)+\alpha \mu(u_\alpha)}.
\end{equation}
Eventually,  by comparing \eqref{eq:67} with \eqref{eq:B2} and using \eqref{eq:61},  we can conclude that $\tilde{R}_{ij}(\alpha)=\tilde{R}_\alpha f(i)=\Psi_{ij}(\alpha)$.  This completes the proof. 
\end{proof}

\section{Birth and death processes with small non-local jumps}\label{SEC9}

Finally, we turn our attention to the ``pathological" case where $|\nu|=\infty$.  
The resolvent matrix for this case can still be explicitly expressed as \eqref{eq:B2}. 
However, the probabilistic construction outlined in Theorem~\ref{THM61} is no longer applicable.
{The goal of this section is to build a sequence of piecing out $Q$-processes that converges to a pathological one in certain aspects.   But it should be noted that our following strategy can only be applied when $\infty$ is regular, as it necessitates the selection of the strictly positive parameters $\beta_n$. 
} 

More specifically,  let us consider a sequence of positive measures $\nu^n$,  $n\geq 1$,  on $\bN$ such that 
\[
	|\nu^n|:=\sum_{k\geq 0}\nu^n_k<\infty,  \quad 0\leq \nu^n_k\leq \nu_k \text{ and }\nu^n_k\uparrow \nu_k \;(n\uparrow \infty),\; k\geq 0.  
\]
A simple example is $\nu^n_k=\nu_k$ for $k\leq n$ and $\nu^n_k=0$ for $k>n$.  
If $\beta>0$,  we let $\beta_n:=\beta$.  Otherwise,  we select a sequence of positive constants $\beta_n>0$ such that $\beta_n\downarrow 0$. 
Set
\begin{equation}\label{eq:68}
	\kappa_n:=\frac{|\nu^n|+\gamma}{\beta_n},\quad \pi_n(\{k\})=\frac{\nu^n_k}{|\nu^n|+\gamma}, \; k\in \bN,\quad \pi_n(\{\partial \})=\frac{\gamma}{|\nu^n|+\gamma}.
\end{equation}
We first demonstrate that the transition matrix corresponding to $(\nu^n, \gamma,\beta_n)$ converges to $(p_{ij}(t))_{i,j\in \bN}$ in a uniform sense.   

\begin{lemma}\label{LM91}
{Assume that $\infty$ is regular. } Let $(p_{ij}(t))_{i,j\in\bN}$ be a Feller $Q$-process with the triple $(\nu, \gamma,\beta)$ such that $|\nu|=\infty$.  Then, for any $j\in \bN$,
\begin{equation}\label{eq:91}
p^n_{ij}(t)\rightarrow p_{ij}(t),\quad n\rightarrow\infty
\end{equation}
uniformly in $(t,i)\in [0, T]\times \bN$ for any fixed constant $T>0$,  where $(p^n_{ij}(t))_{i,j\in \bN}$ is the piecing out of $X^{\kappa_n}$ with respect to $\pi_n$ and $\kappa_n,\pi_n$ are given by \eqref{eq:68}. 
\end{lemma}
\begin{proof}
Note that both $p_{ij}(t)$ and $p_{ij}^n(t)$ are Feller processes on $\bN\cup\{\infty\}\cup\{\partial\}$.  Using \eqref{eq:B2} and Remark~\ref{RMB2},  one can easily obtain that for any $f\in C(\bN\cup\{\infty\})$, 
\[
R^n_{\alpha}f(i)\rightarrow R_\alpha f(i),\quad n\rightarrow\infty
\]
uniformly in $i\in \bN\cup\{\infty\}$,  where $(R^n_\alpha)$ is the resolvent of $(p^n_{ij}(t))$.  Then,  the conclusion follows from the Trotter-Kato theorem; see, e.g.,  \cite[IX,  Theorem~2.16]{K95}. 
\end{proof}

In 1958, Wang (see \cite{WY92}) constructed a sequence of Doob processes that converges to an arbitrary fixed $Q$-process in the same sense as \eqref{eq:91}, although not uniformly. This convergence implies the convergence of the finite-dimensional distributions of the corresponding $Q$-processes.  Regarding our strategy,  since both $(p^n_{ij}(t))$ and $(p_{ij}(t))$ in \eqref{eq:91} are Feller $Q$-processes,  we can achieve another stronger form of convergence. 

Let $X^n:=(X^n_t)_{t\geq 0}$ (resp.  $X=(X_t)_{t\geq 0}$) represent the Feller $Q$-process on $\bN\cup\{\infty\}\cup\{\partial\}$ with the transition matrix $(p_{ij}^n(t))$ (resp. $(p_{ij}(t))$).  
These processes can be realized on the Skorohod space $D_{\bN\cup\{\infty\}\cup\{\partial\}}([0,\infty))$, which consists of all c\`adl\`ag functions on $\bN\cup\{\infty\}\cup \{\partial\}$.  Note that $\bN\cup\{\infty\}\cup\{\partial\}$ is equipped with a certain metric that  conforms to the topology of $\bN\cup\{\infty\}\cup \{\partial\}$,  resulting in $D_{\bN\cup\{\infty\}\cup \{\partial\}}([0,\infty))$ becoming a complete separable metric space; see, e.g., \cite[\S3 Theorem~5.6]{EK05}.  Consequently, $X^n$ (resp.  $X$) can be identified with a specific probability measure $\mathbf{P}^n$ (resp.  $\mathbf{P}$) on $D_{\bN\cup\{\infty\}\cup\{\partial\}}([0,\infty))$.  

\begin{theorem}\label{THM92}
Adopt the same notations as in Lemma~\ref{LM91}. 
 Assume that $X^n_0$ converges to $X_0$ as $n\rightarrow \infty$ in distribution.  Then, $\mathbf{P}^n$ converges weakly to $\mathbf{P}$ as $n\rightarrow \infty$,  i.e.,  for any bounded and continuous function $F$ on $D_{\bN\cup\{\infty\}\cup \{\partial\}}([0,\infty))$,
\[
	\lim_{n\rightarrow \infty} \int_{D_{\bN\cup\{\infty\}}([0,\infty))} F(\omega)\mathbf{P}^n(d\omega)=\int_{D_{\bN\cup\{\infty\}}([0,\infty))} F(\omega)\mathbf{P}(d\omega).  
\]
\end{theorem}
\begin{proof}
Denote by $P^n_t$ and $P_t$ the transition semigroups of $X^n$ and $X$, respectively.  The argument in the proof of Lemma~\ref{LM91} shows that for any $f\in C(\bN\cup\{\infty\}\cup\{\partial\})$,  $P_t^nf$ converges to $P_tf$ in $C(\bN\cup\{\infty\}\cup \{\partial\})$ as $n\rightarrow \infty$.  (Note that $P^n_tf(\partial)=P_tf(\partial)=f(\partial)$.) Then, the result follows by applying \cite[\S4 Theorem~2.5]{EK05}. 
\end{proof}

\appendix

\section{Piecing out of Ikeda-Nagasawa-Watanabe}\label{APPA}

Let $S$ be a locally compact Hausdorff space with a countable base and  let $S_\Delta:=S\cup \{\Delta\}$ be the one-point compactification of $S$ (where $\Delta$ is attached as an isolated point if $S$ is compact).  
We are given a right continuous strong Markov process 
\[
	X=\left\{\Omega, \sF, \sF_t,  X_t , \theta_t,    (\mathbf{P}_x)_{x\in S_\Delta} \right\}
\]
on $S_\Delta$ with the lifetime $\zeta$,  where $\Delta$ is the trap point.  
Consider a probability measure $\pi$ on $S_\Delta$.  Define a kernel $\nu(\omega, dy)$ on $\Omega\times S_\Delta$ as follows:
\begin{equation}\label{eq:A1}
	\nu(\omega,  dy):=\pi(dy),\; \omega\neq \omega_\Delta,\quad \nu(\omega_\Delta,  dy):=\delta_\Delta(dy),
\end{equation}
where $\omega_\Delta\in \Omega$ such that $X_t(\omega_\Delta)=\Delta$ for any $t\geq 0$.  

The piecing out transformation proposed by Ikeda et al. in \cite{INW66} can be seen as, in some sense,  the inverse of killing.  Following \cite{INW66}, let $W:=\Omega \times S_\Delta$ with the $\sigma$-algebra $\mathcal{B}(W):=\sF\otimes \mathcal{B}(S_\Delta)$, and for any $\mathtt{w}=(\omega, y)\in W$, define
\begin{equation}\label{SQ2XTW}
\dot{X}_t(\mathtt{w}):=\left\lbrace
\begin{aligned}
	X_t(\omega),\quad t<\zeta(\omega), \\
	y,\quad\quad t\geq \zeta(\omega). 
\end{aligned}\right.
\end{equation}
For each $x\in S_\Delta$, we can introduce a probability measure $\mathbf{Q}_x(d\mathtt{w}):=\mathbf{P}_x(d\omega)\nu(\omega, dy)$ on $W$. Furthermore,  let $(\tilde{\Omega}, \tilde{\sF})$ be the product of infinitely many countable copies of $(W,\mathcal{B}(W))$.  There exists a unique probability measure $\tilde{\mathbf{P}}_x$ on $(\tilde{\Omega}, \tilde{\sF})$ such that for any $n\geq 1$, 
\begin{equation}\label{eq:A3}
\tilde{\mathbf{P}}_x[ d\mathtt{w}_1,\cdots,d\mathtt{w}_n]
 =\mathbf{Q}_x[d\mathtt{w}_1]\mathbf{Q}_{y_1}[d\mathtt{w}_2]\cdots\mathbf{Q}_{y_{n-1}}[d\mathtt{w}_n],
\end{equation}
where $\mathtt{w}_i=(\omega_i, y_i)$ for $1\leq i\leq n$.  We define a new trajectory for $\tilde{\mathtt{w}}=(\mathtt{w}_1,\cdots, \mathtt{w}_n,\cdots)\in \tilde{\Omega}$ as follows:
\[
\tilde{X}_t(\tilde{\mathtt w})= \begin{cases}
 \dot{X}_t(\mathtt{w}_1), &  \mbox{if}~0\leq t\leq \zeta(\omega_1), \\
 \cdots & \\
 \dot{X}_{t-(\zeta(\omega_1)+\ldots+\zeta(\omega_n))}(\mathtt{w}_{n+1}), & \mbox{if}~\sum\limits_{i=1}^n\zeta(\omega_i)<t\leq \sum\limits_{i=1}^{n+1}\zeta(\omega_i),\\
 \cdots & \\
 \Delta & \mbox{if}~t\geq\tilde{\zeta}(\tilde{\mathtt{w}}):= \sum\limits_{i=1}^{N(\tilde{\mathtt{W}})}\zeta(\omega_{i}),
 \end{cases} 
\]
where $N(\tilde{\mathtt{w}})=\inf\{i:\zeta(\omega_i)=0\}$ with $\inf\emptyset:=\infty$.  (Intuitively, $\tilde{X}$ returns to the state space immediately according to the distribution $\pi$ after blowing up.)
By further defining the shift operators $\tilde{\theta}_t$ and the filtration $\tilde{\sF}_t$ on $\tilde{\Omega}$,  the main result of \cite{INW66} states that
\begin{equation}\label{SQ2XOF}
	\tilde{X}=\left\{\tilde{\Omega},\tilde{\sF}, \tilde{\sF}_t, \tilde{X}_t, \tilde{\theta}_t,  (\tilde{\mathbf{P}}_x)_{x\in S_\Delta}\right\}
\end{equation}
is a right continuous strong Markov process on $S_\Delta$ with the lifetime $\tilde{\zeta}$.  The kernel $\nu$ is called the \emph{instantaneous distribution} of the piecing out transformation in \cite{INW66}. 

For convenience,  we denote $(X, \pi)$ as the \emph{piecing out} $\tilde{X}$ of $X$ with respect to the probability measure $\pi$, which determines the instantaneous distribution $\nu$.  

\section{Analytic construction of birth and death processes}\label{APPB}

 Historically, there has been a well-known analytic approach to characterize all birth and death processes by solving their resolvents.  This method was first introduced by Feller \cite{F59} and later completed by Yang in 1965 (see \cite[Chapter 7]{WY92}).

Adopt the notations presented in Section~\ref{SEC2} {and assume that $\infty$ is either regular or exit}. 
Let $(p_{ij}(t))_{i,j\in \bN}$ be a $Q$-process that is not the minimal one. 
Recall that $u_\alpha(i)=\mathbf{E}_i e^{-\alpha \eta}$ for $i\in \bN$ and denote by $(R^\text{min}_\alpha)_{\alpha>0}$ the resolvent of the minimal $Q$-process.  Set
\[
\Phi_{ij}(\alpha):=R^\text{min}_\alpha(i, \{j\}),\quad \alpha>0,  i,j\in \bN. 
\] 
Let $\nu$ be a positive measure on $\bN$ and let $\gamma, \beta\geq 0$ be two constants.  Set $|\nu|:=\sum_{k\geq 1}\nu_k$.  {When both 
\begin{equation}\label{eq:B1}
	\sum_{k\geq 0} \nu_k \left(\sum_{j=k}^\infty (c_{j+1}-c_j) \sum_{i=0}^j \mu_i \right)<\infty,\quad |\nu| +\beta\neq 0, 
\end{equation}
and
\begin{equation}\label{eq:B4}
\beta=0,\quad \text{if }\infty\text{ is exit}
\end{equation}
are satisfied},  define, for $\alpha>0$ and $i,j\in \bN$,
\begin{equation}\label{eq:B2}
	\Psi_{ij}(\alpha):=\Phi_{ij}(\alpha)+u_\alpha(i) \frac{\sum_{k\geq 0} \nu_k \Phi_{kj}(\alpha)+\beta\mu_j u_\alpha(j)}{\gamma+\sum_{k\geq 0}\nu_k(1-u_\alpha(k))+\beta\alpha\sum_{k\geq 0}\mu_k u_\alpha(k)}.  
\end{equation}
The matrix $(\Psi_{ij}(\alpha))_{i,j\in \bN}$ is called the \emph{$(Q, \nu, \gamma, \beta)$-resolvent matrix. } For a constant $M>0$,  the $(Q, M\nu,  M\gamma, M\beta)$-resolvent matrix is obviously the same as the $(Q, \nu, \gamma,\beta)$-resolvent matrix.  

The following theorem is attributed to \cite[\S7.6]{WY92}. 

\begin{theorem}\label{THMB1}
The transition matrix $(p_{ij}(t))_{i,j\in \bN}$ is a $Q$-process that is not the minimal one,  if and only if there exists a unique (up to a multiplicative positive constant) triple $(\nu, \gamma, \beta)\geq 0$ with \eqref{eq:B1} and \eqref{eq:B4} such that the resolvent of $(p_{ij}(t))_{i,j\in \bN}$ is given by
\[
	R_\alpha(i,\{j\})=\Psi_{ij}(\alpha),\quad \alpha>0,  i,j\in \bN,
\]
where $(\Psi_{ij}(\alpha))_{i,j\in \bN}$ is the $(Q, \nu, \gamma, \beta)$-resolvent matrix.  Furthermore,
\begin{itemize}
\item[(1)] $(p_{ij}(t))_{i,j\in \bN}$ is honest,  if and only if $\gamma=0$.  
\item[(2)] $(p_{ij}(t))_{i,j\in \bN}$ is the $(Q,\pi)$-Doob process with $\pi\neq \delta_\partial$,  if and only if $0<|\nu|<\infty,  \beta=0$ and 
\begin{equation}\label{eq:B3}
\pi(\{k\})=\frac{\nu_k}{\gamma+|\nu|},\; k\in \bN,\quad \pi(\{\partial\})=\frac{\gamma}{\gamma+|\nu|}. 
\end{equation}
\end{itemize}
\end{theorem}
\begin{remark}\label{RMB2}
The first inequality in \eqref{eq:B1} amounts to $\sum_{k,j\in \bN}\nu_k \Phi_{kj}(\alpha)<\infty$ for any $\alpha>0$; see \cite[\S7.10]{WY92}.  
\end{remark}

When $|\nu|+\beta=0$,  \eqref{eq:B2} simplifies to the minimal resolvent matrix $\Phi_{ij}(\alpha)$.  Hence,  we may say that the triple $(0, \gamma, 0)$ for any $\gamma\geq 0$ determines the resolvent matrix of the minimal $Q$-process.

\subsection*{Acknowledgements}

{
The author would like to express his gratitude to the referees for their valuable comments, which have greatly improved the quality of this paper.
The author would like to thank Professor Patrick J. Fitzsimmons.  It was through his suggestion that the author started to learn about the theory of Ray-Knight compactification. The author would also like to thank Professor Jiangang Ying for his guidance and many helpful suggestions.}

\bibliographystyle{abbrv}
\bibliography{RKbdSPA}

\end{document}